\documentclass[10pt,a4paper,reqno]{amsart} 
\usepackage[final]{optional}
\usepackage[british,american]{babel}
\usepackage[margin=0cm]{geometry}
\usepackage{lipsum}
\usepackage[leqno]{amsmath}
\usepackage{mathrsfs}

\usepackage{amsfonts, amsmath, wasysym, caption}

\usepackage{amssymb,amsthm, paralist}

\usepackage{amsopn}
\usepackage{accents,bbm,bibgerm,dsfont,eucal,esint,paralist,url,verbatim,wasysym}
\usepackage[normalem]{ulem}
\usepackage{bbold}

\usepackage{
latexsym,
nicefrac
}

\usepackage[usenames]{color}
\usepackage{tikz}
\usetikzlibrary{decorations.pathreplacing}
\usepackage{url}

\definecolor{darkgreen}{rgb}{0,0.5,0}
\definecolor{darkred}{rgb}{0.7,0,0}

\usepackage[colorlinks, 
citecolor=darkgreen, linkcolor=darkred
]{hyperref}

\usepackage{esint}
\usepackage{bibgerm}
\usepackage[normalem]{ulem}

\usepackage{enumitem}
\theoremstyle{plain}
\newtheorem{lemma}{Lemma}[section]
\newtheorem{thm}[lemma]{Theorem}

\newtheorem{cor}[lemma]{Corollary}
\theoremstyle{definition}

\newtheorem{rmk}[lemma]{Remark}

\numberwithin{equation}{section}

\newcommand{\al}{\alpha}

\newcommand{\de}{\delta}
\newcommand{\De}{\Delta}

\newcommand{\Om}{\Omega}

\newcommand{\la}{\lambda}
\newcommand{\La}{\Lambda}

\newcommand{\si}{\sigma}

\renewcommand{\th}{\theta}

\newcommand{\ep}{\varepsilon}

\newcommand{\R}{\ensuremath{{\mathbb R}}}
\newcommand{\N}{\ensuremath{{\mathbb N}}}

\newcommand{\C}{\ensuremath{{\mathbb C}}}

\newcommand{\norm}[1]{\Vert#1\Vert}

\def\osc{\mathop{{\mathrm{osc}}}\limits}

\newcommand{\brmk}{\begin{rmk}}
\newcommand{\ermk}{\end{rmk}}
\newcommand{\partref}[1]{\hbox{(\csname @roman\endcsname{\ref{#1}})}}

\newcommand{\beq}{\begin{equation}}
\newcommand{\eeq}{\end{equation}}
\newcommand{\beqs}{\begin{equation*}}
\newcommand{\eeqs}{\end{equation*}}
\newcommand{\beqa}{\begin{equation}\begin{aligned}}
\newcommand{\eeqa}{\end{aligned}\end{equation}}
\newcommand{\beqas}{\begin{equation*}\begin{aligned}}
\newcommand{\eeqas}{\end{aligned}\end{equation*}}

\newcommand{\half}{\frac{1}{2}}
\newcommand{\thalf}{\tfrac{1}{2}} 
\newcommand{\dist}{\text{\textnormal{dist}}}
\renewcommand{\i}{\mathrm{i}}
\newcommand{\Mneu}{\widetilde{\mathcal{M}}}
\newcommand{\pt}{\partial_t}
\newcommand{\ddt}{\tfrac{d}{dt}}

\newcommand{\abs}[1]{{\vert#1\vert} }
\newcommand{\babs}[1]{{\big\vert#1\big\vert} } 
\newcommand{\eps}{\varepsilon}
\newcommand{\na}{\nabla}

\newcommand{\ddeps}{\tfrac{d}{d\eps}\vert_{\eps=0}}

\newcommand{\supp}{\text{supp}} 
\newcommand{\nto}{\nrightarrow}

\newcommand{\Var}{\mathcal{V}}

\newcommand{\Arg}{\text{Arg}}
\newcommand{\lan}{\langle}
\newcommand{\ran}{\rangle}
\newcommand{\Ima}{\text{Im}}
\newcommand{\Rea}{\text{\textnormal{Re}}}

\newcommand{\dbar}{\bar\partial}
\newcommand{\tens}{\Delta} 
\newcommand{\res}{\text{\textnormal{res}}} 

\title{{\sc
Analysis of boundary bubbles for almost minimal cylinders
}
\\ 
}
\author{Melanie Rupflin and Matthew R. I. Schrecker}
\date{\today}

\begin{document}
\begin{abstract}
We analyse the asymptotic behaviour of solutions of the Teichm\"uller harmonic map flow from cylinders, and more generally of `almost minimal cylinders', in situations where the maps satisfy a Plateau-boundary condition for which the three-point condition degenerates. 
We prove that such a degenerating boundary condition forces the domain to stretch out as a boundary bubble forms. 
Our main result then establishes that for prescribed boundary curves that satisfy Douglas' separation condition, these boundary bubbles 
will not only be harmonic but will themselves be branched minimal immersions. 
Together with earlier work, this in particular completes the proof that the Teichm\"uller harmonic map flow changes every initial surface in $\R^n$ spanning such boundary curves into a solution of the corresponding Douglas-Plateau problem. 
\end{abstract}

\maketitle

\section{Introduction and results}
Let $\Gamma^{\pm}$ be two disjoint $C^3$ Jordan curves in Euclidean space $\R^n$, $n\geq 3$, and let $u_0:C_0\to \R^n$ be any given map from the cylinder $C_0=[-1,1]\times S^1$ that spans $\Gamma^\pm$ in the sense that $u_0\vert_{\{\pm 1\}\times S^1}$ is a weakly monotone parametrisation of $\Gamma^\pm$. 

We consider a geometric flow, introduced by the first author in \cite{R-cyl}, that is designed to change such an initial map into a parametrisation of a minimal surface spanning $\Gamma^\pm$. 
This flow 
is modelled on the construction of Teichm\"uller harmonic map flow
for maps $u$ from arbitrary closed surfaces  $M$  into general  target manifolds $(N,g_N)$ 
 by P. Topping and the first author from \cite{RT}. The flow in \cite{RT} is defined as a natural gradient flow of the Dirichlet energy
$$E(u,g)=\half \int_M\abs{du}_g^2 \,dv_g$$
viewed as a function on the space of equivalence classes of pairs $(u,g)$ of maps $u:M\to (N,g_N)$ and metrics $g$ on $M$, identified under the symmetries of $E$, see \cite{RT} for details. The results of P. Topping and the first author \cite{RT}-\cite{RT-global}  establish that this flow decomposes every initial map $u_0:M\to N$ from a closed domain to a compact Riemannian manifold into critical points of the area functional, i.e.~into (possibly branched) minimal immersions. 

For the analogous flow of maps from cylinders to Euclidean space satisfying a Plateau-boundary condition, the results of \cite{R-cyl} establish the existence of a global weak solution for every initial data. In the case that the corresponding \textit{three-point condition} \eqref{def:three-standard} does not degenerate as $t\to \infty$, the results of \cite{R-cyl} furthermore ensure convergence along a sequence $t_i\to \infty$ to 
a critical point of the area functional that spans the prescribed boundary curves $\Gamma^\pm$ in the target, so that in these situations the flow changes the initial surface into (possibly branched) minimal immersions as desired. 

Here we investigate the remaining case of the asymptotic behaviour of the flow, i.e.~the case of a degeneration of the three-point condition. This is equivalent to more and more of the prescribed boundary curves $\Gamma^\pm$ in the target being parametrised over smaller and smaller subarcs of the corresponding boundary curves   of the domain. As we shall see, this forces \textit{harmonic bubble(s)} to form at the boundary of the domain, a phenomenon that is excluded for many similar geometric flows, such as the harmonic map flow with Dirichlet boundary condition or the flows of Chang-Liu \cite{CL1, CL2,CL3}, respectively Duzaar-Scheven \cite{D-S}, that change disc-type surfaces into minimal discs, respectively discs with prescribed mean curvature. 

While the bubbling behaviour of almost harmonic maps on closed domains has been the subject of intensive study over the past decades, see e.g. \cite{bubbles-1,bubbles-2,bubbles-3,Topping-bubbles,bubbles-Miaomiao} and the references therein,
 relatively little has been known about the formation and structure of boundary bubbles until very recently, when work of Jost, Liu and Zhu \cite{JLZ},  respectively Huang, Wang \cite{HW}, established the  energy identity and no-neck property 
for the bubble-tree convergence of approximate harmonic maps satisfying free-boundary, respectively weak and strong anchoring conditions, see also \cite{Laurain}.

While the present work is also concerned with the analysis of boundary bubbles, its focus is quite different as our main goal is to establish minimality of the obtained harmonic bubbles and as we deal with Plateau-boundary data for which a (degenerating) three-point condition is imposed. In contrast to free-boundary harmonic discs, which are always (branched) minimal immersions, the bubbles we obtain are in general not conformal and can indeed  a priori have poles in no fewer than four points. Excluding the formation of such poles will be a key step in the asymptotic analysis of the flow, and more generally the analysis of almost minimal maps.

For boundary curves that satisfy the \textit{separation condition}, which we introduce in \eqref{ass:no-linking} below and 
which is a weaker version of Douglas' separation condition that we recall in \eqref{eq:Douglas}, we obtain that the resulting bubbles formed by the flow are indeed minimal immersions provided the three-point condition does not degenerate too quickly, a condition which can be easily guaranteed by a suitable coupling of the equations of the flow, see \eqref{ass:coupling} below.

As a consequence of our analysis and the results of \cite{R-cyl}, we thus obtain in particular

\begin{thm}
Let $\Gamma^\pm\subset \R^n$ be any two disjoint $C^3$ Jordan curves satisfying the separation condition \eqref{ass:no-linking} and let $u_0\in H^1(C_0,\R^n)$ be any initial map which spans $\Gamma^{\pm}$ in the sense that $u_0\vert_{\{\pm1\}\times S^1}$ is a (weakly monotone) parametrisation of $\Gamma^\pm$, and let $g_0\in  \Mneu$ be any initial metric. 
\\
Then the global solution $(u,g)$ of the flow 
\eqref{eq:fluss-metric} and \eqref{eq:fluss-map} for coupling functions $\eta_\pm$ satisfying \eqref{ass:coupling} changes $u_0$ into either a minimal cylinder or into two minimal discs in the sense that there exist times $t_i\to \infty$ so that $u(t_i)$ converges to
(branched) minimal immersions 
$u_\infty:C_0\to \R^n$, respectively $\hat u_\infty^\pm:D\to \R^n$, 
spanning $\Gamma^\pm$ as is made precise in \cite[Theorem 2.7]{R-cyl}, respectively Theorem \ref{thm:1} below.
\end{thm}

To state our main results in full detail, we first need to recall more details on the construction and properties of the flow from \cite{R-cyl}. As in the case of a closed domain, for $M=C_0$ a cylinder, the flow is defined as a gradient flow of the energy $E$ on the set of pairs of maps and metrics that are identified under the symmetries of $E$, i.e.~conformal invariance $E(u,g)=E(u,e^{2v} g)$ and invariance $E(u,g)=E(u\circ  f,f^*g)$ under 
pull-back by diffeomorphisms $f:C_0\to C_0$.

We recall from \cite{R-cyl} that the conformal invariance allows us to restrict the metric to the set of hyperbolic metrics with boundary curves of prescribed constant geodesic curvature. We also recall that if we were to use \textit{all} of the symmetries to simplify the evolution of the metric component (rather than use some of the symmetries for the map component instead), we could further restrict the metric to be an element of the one-parameter family $G_\ell$ of metrics which are obtained as follows:
for each $\ell\in(0,\infty)$, we consider the standard hyperbolic collar  
\beq
\label{def:cyl-ell}
(C_\ell, g_\ell)=\big([-Y_\ell,Y_\ell]\times S^1,\rho_\ell^2(s)(ds^2+d\th^2)\big),
\eeq
where
\beq\label{def:collar-metric}
\rho_\ell(s)=\tfrac{\ell}{2\pi}\cos(\tfrac{\ell}{2\pi}s)^{-1}, \text{ } Y_\ell=\tfrac{2\pi}{\ell}\big(\tfrac{\pi}{2}-\arctan(c\ell)\big),
\eeq
for a fixed $c>0$. The $G_\ell$ are now obtained as pull-back of $g_\ell$
by diffeomorphisms $f_\ell:C_0\to [-Y_\ell,Y_\ell]\times S^1$ 
which are chosen such that $(G_\ell)$ is horizontal, i.e.~moves $L^2$-orthogonal to the action of the diffeomorphisms, and whose precise form was given in \cite[Lemma 2.4]{R-cyl}. This condition on $G_\ell$ is equivalent to 
$\tfrac{d}{d\ell}G_\ell$ being an element of the horizontal space $H(G_\ell)$
which is made up of all 
 trace- and divergence-free tensors $k$ for which furthermore $k(\nu,t)\equiv 0$ on $\partial C_0$, if $\nu,t$ are  normal, respectively tangent, to $\partial C_0$, so 
\beqs 
H(G_\ell):=\big\{ f_\ell^* \big(c ( ds^2-d\th^2)\big), c\in \R\big\}.\eeqs
While for the flow of Topping and the first author \cite{RT} on closed domains, the metric component only moves in horizontal directions, for surfaces with boundary, we need to reserve some of the symmetries coming from the diffeomorphism invariance to impose a three-point condition on the map component, compare \eqref{def:three-standard} below. As a result, one must also allow the metric to move by the pull-back by select diffeomorphisms.
In the case of the cylinder, these are given by 
elements of the six-parameter family of diffeomorphisms 
$$(h_{b,\phi})_{(b,\phi)=((b^+,\phi^+),(b^-,\phi^-))\in \Om^2}, \text{ where }\Om=D\times \R\subset  \C\times \R, \quad D=D_1(0),$$ 
whose precise definition is given in equation (4.2) in \cite[Section 4.1.1]{R-cyl}. The key properties of these diffeomorphisms are that 
their restrictions onto the boundary curves 
$\{\pm1\}\times S^1$ are described by the corresponding M\"obius transforms, i.e.~their angular component is such that 
\beqs
e^{\i h^\th_{(b,\phi)}(\pm 1,\th)}= M_{b^\pm,\phi^\pm}(e^{\i \th}), \text{ where } M_{b^\pm,\phi^\pm}(z)=e^{\i \phi^\pm}\tfrac{z+b}{1+\overline{b} z},\eeqs
and that a change of the metric 
$$\Var^{+}(g):=\{\ddeps h_{(b^+,\phi^+)(\eps), (b^-,\phi^-)}^* G_{\ell} \text{ where } (b^+,\phi^+)\vert_{\eps=0}=(b^+,\phi^+) \text{ and } g=h_{(b,\phi)}^*G_\ell\}, $$
induced by a change of only $(b^+,\phi^+)$
is supported in a fixed compact subset of the corresponding half-cylinder $C^+:=(0,1]\times S^1$, 
with the analogous statement holding also for  variations $\Var^-(g)$ induced by changes of $(b^-,\phi^-)$.

The metric component of the flow hence evolves in
$\Mneu:=\{ h_{b,\phi}^*G_\ell:\, (b,\phi)\in\Om^2,\, \ell\in (0,\infty)\}$
whose tangent space splits $L^2$-orthogonally as 
\beq
\label{eq:splitting1}
T_g\Mneu=H(g)\oplus \Var^+(g)\oplus \Var^-(g).\eeq
 We can equip $T_g\Mneu$ with either the $L^2$-metric or, more generally, with a product metric  
$$\norm{h+L_{X_+}g+L_{X_-}g}^2:= \norm{h}_{L^2(C_0,g)}^2+\eta_+^{-2}\norm{L_{X_+}g}_{L^2(C_0,g)}^2+\eta_-^{-2}\norm{L_{X_-}g}_{L^2(C_0,g)}^2,
$$
for $ h\in H(g)$ and $L_{X_{\pm}}g\in\Var^\pm(g)$,
where we shall assume in the following that $\eta_{\pm}=\eta(\abs{b^\pm})$ for a smooth function $\eta:[0,1)\to \R^+$ which is bounded from above.

As the negative $L^2$-gradient of $E$ is described in terms of the Hopf-differential $\Phi(u,g)$ by
$$-\na_gE=\tfrac14 \Rea(\Phi(u,g)), \text{ where } \Phi(u,g)=\big[\abs{u_x}^2-\abs{u_y}^2-2\i \langle u_x,u_y\rangle \big] dz^2, \,z=x+\i y$$
for local isothermal coordinates $(x,y)$ on $(C_0,g)$,
the resulting gradient flow for the metric component is hence characterised by 
\beq \label{eq:fluss-metric}
\pt g=\tfrac14 \big[ P_g^{H}(\Rea(\Phi(u,g)))+\eta^2(\abs{b^+})\cdot  P_g^{\Var^+}(\Rea(\Phi(u,g))) +\eta^2(\abs{b^-})\cdot P_g^{\Var^-}(\Rea(\Phi(u,g))) \big],  \eeq
where $P_g^H$, $P_g^{\Var^\pm}$ refer to the $L^2$-orthogonal projections onto $H(g)$, respectively $\Var^\pm(g)$.

The map component of a solution of the flow is determined by evolving a map $u$ in $H^1_{\Gamma,*}(C_0,\R^n)$ by the negative $L^2$-gradient of the energy, and hence characterised by 
\beq
\label{eq:fluss-map}
\int_{[0,T]\times C_0}\lan du, dw\ran_{g}+ \pt u\cdot w\, dv_{g} \, dt \geq 0 
\text{ for all }w\in L^2([0,T],T^+_u H^1_{\Gamma,*}(C_0))
\eeq
where $H^1_{\Gamma,*}(C_0)$ is the set of all functions $u\in H_{\Gamma}^1(C_0)$, where
\beqs 
H^1_{\Gamma}(C_0):=\{u\in H^1(C_0,\mathbb{R}^n), \text{ with } u\vert_{\{\pm1\}\times S^1} \text{ a weakly monotone parametrisation of }\Gamma^\pm\}\eeqs
which furthermore satisfy the three-point condition 
\beq
\label{def:three-standard}
u(P_0^{j,\pm})=Q^{j,\pm} , \quad j=1,2,3, \text{ where } P_0^{j,\pm}=(\pm1, \tfrac{2j\pi}{3})\eeq
for some fixed distinct points $Q^{j,\pm}\in \Gamma^\pm$. Here $T_u^+H^1_{\Gamma,*}(C_0)$ is the corresponding tangent cone, compare \cite[Section 2.2]{R-cyl}.

\begin{rmk}
We note that the flow considered in \cite{R-cyl} corresponds to the case that the coupling functions $\eta_\pm\equiv 1$, but observe that all results of \cite{R-cyl} remain valid also for general smooth coupling functions $\eta_\pm=\eta(\abs{b^\pm})$ with $\eta>0$ bounded from above, as explained in  Appendix \ref{subsect:app-flow}.
We also note that 
we could furthermore introduce a coupling that relates the speeds of the horizontal part of the metric and the map component, and that questions related to the choice of the coupling constant in the closed case have been considered by Huxol in \cite{Huxol}.
\end{rmk}
 
The results of \cite{R-cyl} guarantee that for every $u_0\in H^1_{\Gamma,*}(C_0)$, there exists a global weak solution
$$(u,g)\in 
 \big(L^\infty([0,\infty),H^1_{\Gamma,*}(C_0,\R^n))\cap H^1([0,\infty)\times C_0)\big)\times C^{0,1}([0,\infty),\Mneu)$$ 
of \eqref{eq:fluss-metric} and \eqref{eq:fluss-map}
 which, furthermore, is stationary at the boundary of $C_0$ except at the points $P_0^{j,\pm}$ for almost every time $t$, compare \eqref{eq:stationar} below,
and which satisfies the energy inequality 
\beqa \label{eq:energy-cond} E(u,g)(t_1)-E(u,g)(t_2)\geq& \thalf \int_{t_1}^{t_2}\norm{\pt u}_{L^2(C_0,g)}^2\,  dt+\int_{t_1}^{t_2}\langle -\na_g E(u,g),\pt g\rangle_{L^2(C_0,g)} \,dt\eeqa
for   almost every  $t_1<t_2$, where by construction 
\beqa \label{eq:energy-term-g}
\langle -\na_g E,\pt g\rangle_{L^2}
=\tfrac14\big[\norm{P_g^{H}(\Rea(\Phi))}_{L^2}^2+\eta_+^2\norm{P_g^{\Var^+}(\Rea(\Phi))}_{L^2}^2 +\eta_-^2\norm{P_g^{\Var^-}(\Rea(\Phi))}_{L^2}^2\big].
 \eeqa
 Here and in the following, we say that a map $u:(\Omega,g)\to\mathbb{R}^n$ satisfies the stationarity condition (or say $u$ is stationary) at a boundary curve $\gamma$ of a domain $\Omega$ except at finitely many points $P^j$
 if there exists a neighbourhood $U$ of $\gamma$ such that
\beq\label{eq:stationar}
\tfrac14\int \Rea( \Phi(u,g))L_X g\, dv_g+\int du(X)\cdot \Delta_g u\, dv_g =0
\eeq
 for every $X\in \Gamma(T\Omega)$ such that $X$ is tangential to $\gamma$, $X(P^j)=0$ and $\supp(X)\subset U$.
 
In the case that the diffeomorphisms $(h_{b,\phi})_t$ by which we pull back the metrics $G_\ell$ to obtain $g=h_{b,\phi}^*G_\ell$ do not degenerate as $t\to \infty$, the results of \cite{R-cyl} ensure that the flow changes the initial surface into either 
 a (branched) minimal immersion from the cylinder spanning both $\Gamma^{\pm}$ (if $\ell(t_i)\nto 0$) or into two (branched) minimal immersions from discs, each spanning one of the boundary curves (if $\ell(t_i)\to 0$), as is described in detail in \cite[Theorem 2.7]{R-cyl}.
 
In the present paper, we address the remaining case of the  asymptotic behaviour, i.e.~the case 
that the diffeomorphisms $h_{b,\phi}$ degenerate as in \eqref{ass:b_to_1}, which will result in the formation of a \textit{harmonic boundary bubble}.
We shall furthermore see that for coupling functions $\eta_\pm=\eta(\abs{b^\pm})$, for $ \eta\in C^\infty([0,1),\R^+)$  so that 
 \beq
\label{ass:coupling} 
 \eta(\xi)\leq C(1-\xi)^\gamma \text{ for some } C\in\R \text{ and }\gamma> \frac{1}{1-\frac{1}{\sqrt{2}}},
 \eeq
 and for boundary curves $\Gamma^\pm$ that satisfy the separation condition 
 \beq
 \label{ass:no-linking}
 \Gamma^\pm\cap S_{\Gamma^\mp}=\emptyset \text{ for every  minimal disc } S_{\Gamma^\mp} \text{ that spans }\Gamma^\mp,
 \eeq
 these boundary bubbles must themselves be (branched) minimal immersions so that the flow will again decompose the initial map into two minimal discs spanning $\Gamma^\pm$. 
 
 We remark that our condition \eqref{ass:no-linking} is a weaker version of Douglas' separation condition that requires  
 \beq \label{eq:Douglas}
 \dist(S^+,S^-)>0 \text{ for all minimal discs } S^\pm \text{ spanning } \Gamma^\pm,
 \eeq
 and ensures the existence of minimal cylinders spanning $\Gamma^\pm$, compare \cite{Douglas}, and under which 
Struwe \cite{Struwe-cyl-paper} developed a Morse theory for 
  the energy of annulus type minimal surfaces, there with the energy viewed as a function of the parametrisations of the boundary curves via harmonic extension.

 To state our result in detail, we first note that 
using the identification $(u,g)\sim (u\circ f, f^*g)$, $f$ a diffeomorphism, we can view 
a solution of the flow at times $t$  
equivalently as either maps 
$u(t)$ that satisfy the standard three-point condition
\eqref{def:three-standard} but are defined on cylinders $(C_0, g(t)=h_{(b,\phi)(t)}^*f_{\ell(t)}^*g_{\ell(t)})$ whose metrics become singular as $\abs{b^\pm}\to 1$, or 
 as maps 
\beq \label{def:utilde}
\tilde u(t):= u(t)\circ (f_{\ell(t)} \circ h_{(b,\phi)(t)})^{-1}
:(C_{\ell(t)}, g_{\ell(t)} )\to \R^n, 
\eeq
on the well controlled cylinders $(C_{\ell(t)}, g_{\ell(t)} )$ 
that are described by \eqref{def:cyl-ell}, 
 at the cost of the new maps now satisfying a three-point condition that degenerates as $\abs{b^\pm}\to 1$: namely for every $t\in [0,\infty)$ the map $\tilde u(t)$ satisfies 
$
\tilde u(t,P^{j,\pm})= Q^{j,\pm}
$
for the same fixed points $Q^{j,\pm}\in \Gamma^\pm$  in the target, but now for points $P^{j,\pm}(t)=(\pm Y_{\ell(t)},\th^{j,\pm}(t)) $ which are characterised by 
\beq\label{def:P^j,pm}
e^{\i \th^{j,\pm}(t)}=M_{b^\pm(t),\phi^\pm(t)}(e^{\i \frac{2\pi}{3} j}).
\eeq
If $\abs{b^\pm(t)}\to 1$, 
then $\eps^\pm:=\min_{j\neq k}(\dist_{g_{\ell(t)}}(P^{k,\pm}(t),P^{j,\pm}(t)))\to 0$ so, after passing  to a subsequence $t_i\to \infty$, 
at least two of the points $P^{j,\pm}(t_i)$ will converge to the same accumulation point which we denote by 
$P_*^\pm$. As we shall see below, this not only forces a boundary bubble to form at $P_*^\pm$, but furthermore forces the cylinder to stretch out and become infinitely long and the map to converge to a constant away from the concentration point(s) $P_*^\pm$.

To be more precise, we shall prove the following result:

\begin{thm}\label{thm:1}
Let $u_0\in H_{\Gamma,*}^1(C_0)$ be any initial map, $g_0\in\Mneu$ any initial metric, and let $(u,g)$ be the corresponding global solution of Teichm\"uller harmonic map flow \eqref{eq:fluss-metric}, \eqref{eq:fluss-map} on the cylinder 
 which is stationary at the boundary except at the three points $P^{j,\pm}_0$ and satisfies the energy inequality \eqref{eq:energy-cond}. Here $\eta_\pm=\eta(\abs{b^\pm})$ for an arbitrary function $\eta\in C^\infty([0,1),\R^+)$ that is bounded from above.
\\
Suppose that the three-point condition degenerates as $t\to \infty$ in the sense that 
\beq
\label{ass:b_to_1}
\liminf_{t\to \infty} \max(\abs{b^+(t)}, \abs{b^-(t)})= 1.
\eeq
Then 
$$\ell(t)\to 0 \text{ as } t\to \infty$$
and there exist $t_i\to \infty$, chosen in particular so that $(b,\phi)(t_i)\to(b_\infty,\phi_\infty)$,
so that the following holds true for the corresponding maps $\tilde u(t_i):(C_{\ell(t_i)},g_{\ell(t_i)})\to \R^n$ that are defined by \eqref{def:utilde}:
\begin{enumerate}
 \item[(i)] 
The maps $u^\pm_i(s,\theta):=\tilde u(t_i)(\pm Y_{\ell(t_i)}\mp s,\theta )$
converge on the half-cylinder 
 $[0,\infty)\times S^1$
to  limit maps $u^\pm_\infty$ 
which are such that
 if $|b^\pm_\infty|=1$, then $u^\pm_\infty$ is constant;
while if $|b^\pm_\infty|<1$, then $u^\pm_\infty$ can be extended across the puncture to a (branched) minimal immersion on the unit disc.
In both cases the convergence is weak $H^2$-convergence away from at most three points and strong $H^1$-convergence away from at most one point, see Theorem \ref{thm:A} for details.
 \item[(ii)] Harmonic bubbles will form at $P^\pm_*$ in the sense that if $|b^\pm_\infty|=1$, then the maps $\tilde{u}(t_i)$ can be pulled back to a sequence of maps $\hat{u}^\pm_i:\Omega_i^\pm\to\mathbb{R}^n$
that converge weakly in $H^2$ away from at most four points to a harmonic map $\hat{u}^\pm_\infty:D\to \R^n$ that spans $\Gamma^\pm$. Here the domains $\Omega_i^\pm$ are such that they exhaust the unit disc minus a point $\widehat P_*^\pm\in \partial D$. 
 \item[(iii)] Moreover, if the coupling function satisfies \eqref{ass:coupling} and the curves $\Gamma^\pm$ satisfy the separation condition \eqref{ass:no-linking}, then the bubbles $\hat{u}^\pm_\infty:D\to\mathbb{R}^n$ are themselves (branched) minimal immersions and indeed no Hopf-differential is lost near the boundary in the sense that 
 $$\norm{\Phi(u_i^\pm)}_{L^1([0,\La]\times S^1)}\to 0 \text{ for every } \La>0.$$
\end{enumerate}
 \end{thm} 

The above result will be a consequence of more general results, stated in Theorems \ref{thm:A} and \ref{thm:B} below, about the formation of boundary bubbles for sequences of \textit{almost minimal maps}
from hyperbolic cylinders 
\beq \label{def:cyl}
(C_i,g_i)=\big([-Y_i^-,Y_i^+]\times S^1, \rho_{\ell_i}^2(ds^2+d\th^2)\big)\text{ for some } \ell_i\in (0,\ell_0), \eeq
where $\rho_{\ell_i}$ are as defined in \eqref{def:collar-metric}, and $Y_i^\pm\geq 1$ are so that
\beq \label{ass:cyl} \tfrac{2\pi}{\ell_i}(\tfrac{\pi}2-c_1\ell_i)\leq Y_i^\pm \leq \tfrac{2\pi}{\ell_i}(\tfrac{\pi}2-c_2\ell_i),\eeq
 for some fixed numbers $\ell_0, c_{1,2}>0$. 
  We note that this
  ensures that the metrics $g_i$ are uniformly equivalent to the standard metric $ds^2+d\th^2$ on `chunks' 
 $K_i^{\pm,\Lambda}:=\{(s,\th): Y^\pm-\La\leq \pm s\leq Y^{\pm}\}$ of fixed length $\La>0$ at the ends of the cylinder and 
remark that \eqref{ass:cyl} is satisfied, in particular, for the cylinders \eqref{def:cyl-ell} arising from the flow considered above. For related results on almost minimal maps from closed (degenerating) surfaces we refer to \cite{HRT,bubbles-Miaomiao}.

\begin{thm}
\label{thm:A}
Let $\Gamma^\pm\subset \R^n$ be two disjoint $C^3$ Jordan curves, 
let $(C_i,g_i)$ be hyperbolic cylinders as in \eqref{def:cyl} and \eqref{ass:cyl} and let $Q^{1,2,3,\pm}\in \Gamma^\pm$ be distinct points.
Let 
$u_i\in H^1_\Gamma(C_i,g_i)$ be a sequence of maps  with uniformly bounded energy
that satisfy the stationarity condition \eqref{eq:stationar} at the boundary curves of $C_i$ except at six points $P_i^{j,\pm}$ for which furthermore 
$u_i(P_{i}^{j,\pm})=Q^{j,\pm}, j=1,2,3. $
Suppose that at least one of these three-point conditions degenerates in the sense that 
 at least one of 
\beq \label{ass:deg-3-point}
\eps_i^\pm:=\min_{j\neq k}(\dist_{g_i}(P_{i}^{j,\pm}, P_{i}^{k,\pm}))\to 0,\eeq
and that the maps $u_i$ are almost 
minimal in the sense that  
\beq \label{ass:almost-min} 
\norm{\tens_{g_i} u_i}_{L^2(C_i,g_i)}+\norm{P^H_{g_i}(\Rea(\Phi(u_i,g_i)))}_{L^2(C_i,g_i)}\to 0,
\eeq
and, in the case that $\eps_i^\pm\not\to 0$, also
\beq
\norm{P_g^{\Var^\pm}(\Rea(\Phi(u_i,g_i)))}_{L^2(C_i,g_i)}\to 0.
\eeq
Then 
$$\ell_i\to 0$$ and, after passing to a subsequence, the maps $u_i$
converge in the following sense:

If one of the three-point conditions does not degenerate, say $\eps_i^-\to 0$ but $\eps_i^+\to\eps_\infty^+>0$,
then the maps $u_i^+:(s,\th)\mapsto u_i(Y^+_i- s,\th)$ converge to a limit $u_\infty^+$
which may be extended across the puncture to a (branched) minimal disc spanning $\Gamma^+$, compare \cite[Theorem 2.7]{R-cyl}.

On the half-cylinders $C_i^\pm=C_i\cap \{ \pm s>0\} $ on which the three-point condition degenerates (which may be one or both of them), i.e.~for which $\eps_i^\pm\to 0$, we have
\begin{enumerate}[label=\textnormal{(\arabic*)}]
 \item[(i)] The maps $u_i^\pm:(s,\th)\mapsto u_i(\pm Y^\pm_i\mp s,\th)$ converge weakly in $H^2_{loc}([0,\infty)\times S^1\setminus \{P_*^\pm\})$ to a constant limiting map $u^\pm_\infty\equiv q^\pm\in\Gamma^\pm$, where
 $P_*^\pm=(0,\th_*^\pm)$ for the mutual limit $\th_*^\pm$ of the angular components of at least two of the three points $P_i^{j,\pm}$.
 \item[(ii)]
 Harmonic bubbles form at $P^\pm_*$ in the sense that the maps
  converge to a harmonic map $\hat{u}_\infty^\pm:D\to\mathbb{R}^n$ spanning $\Gamma^\pm$ after being pulled-back by  suitable conformal diffeomorphisms  $\hat f_i^\pm:\Omega_i^\pm\to [0,Y_i^\pm]\times S^1$, namely
  $$\hat{u}_i^\pm:=u_i^\pm\circ \hat f_i^\pm\to\hat{u}_\infty^\pm \text{ in $H^1_{loc}(D\setminus \{ \widehat{P}_*^\pm\})$ for some }\widehat P_*^\pm \in \partial D $$ 
where the $\hat f_i^\pm$ are defined on a sequence of domains $\Omega_i^\pm$ exhausting $D\setminus\{\widehat{P}_*^\pm\}$ and furthermore converge to the constant map $P_*^\pm$ uniformly away from $\widehat{P}_*^\pm$.
 
 Moreover, the limit maps satisfy the stationarity condition \eqref{eq:stationar} for all vector fields $X \in\Gamma(TD)$ with support $\supp(X)\cap\{\widehat{P}_*^\pm\}=\emptyset$ and with $X(\widehat{P}^{j,\pm})=0$, $j=1,2,3$, where $\widehat{P}^{j,\pm}$ are the limits of the points $(\hat f_i^{\pm})^{-1}(P_i^{j,\pm})$, and the convergence is also weakly in $H^2_{loc}$ away from these four points $\{\widehat{P}^{j,\pm}, \widehat{P}_*^\pm\}$. 
\end{enumerate}
\end{thm}

Our main result, Theorem \ref{thm:B}, then shows that a suitable relation between the decay \eqref{ass:almost-min} of the tension and the horizontal part of the Hopf-differential on the one hand and the rate of the degeneration \eqref{ass:deg-3-point} of the three-point condition on the other hand ensures that the bubbles we obtain are not only harmonic but indeed describe minimal discs.

\begin{thm}
\label{thm:B}
Let $\al, R, E_0>0$, let $\Gamma\subset \R^n$ be a $C^3$ Jordan curve and let $Q^{1,2,3} \in \Gamma$ be distinct points.
Let $u_i\in H^1(C_i,g_i)$ be a sequence of maps with $E(u_i,g_i)\leq E_0$
from  
cylinders $C_i$ which are as in \eqref{def:cyl} and \eqref{ass:cyl} and for which $\ell_i\leq \bar \ell$ for some number $\bar \ell(E_0,\al, c_{1,2})>0$ that is determined by Lemma \ref{lemma:ell-upper}.

Suppose that $u_i|_{\{-Y_i^-\}\times S^1}$ are weakly monotone parametrisations of $\Gamma$ which satisfy 
the stationarity condition \eqref{eq:stationar} at $\{-Y_i^-\}\times S^1$ except at three points $P_i^j$ for which furthermore $u_i(P_i^j)=Q^j$ and denote by $\eps_{i,1}\leq \eps_{i,2}\leq \eps_{i,3}$ the distances between the points $P_{i}^j$. 

If the three-point condition degenerates in the sense that 
$$\eps_{i,1}=\min_{j\neq k} (\dist_{g_i}(P_{i}^j, P_{i}^k))\to 0,$$
and the maps are almost minimal in the sense that
\beq
\label{ass:tension}
\de_i:=\norm{\tens_{g_i}u_i}_{L^2(C_i,g_i)}+\norm{P_{g_i}^H(\Rea(\Phi(u_i,g_i)))}_{L^2(C_i,g_i)}^4\to 0\eeq
and if furthermore the mean values $M_{u_i}(\pm Y_i^{\pm}):=\fint_{\{\pm Y_i^{\pm}\}\times S^1} u_i \, d\th$ of $u_i$ at the ends of the cylinder are so that  
\beq
\label{ass:mv-different}
0<\alpha\leq \abs{M_{u_i}(-Y_i^{-})-M_{u_i}(Y_i^{+})}\leq R,
\eeq
then 
the following holds true:
\begin{itemize}
\item[(i)] The cylinders $C_i$ degenerate at a rate of 
\beq\label{est:upper-ell-thm}
\ell_i\leq C\big(\norm{P_{g_i}^H(\Rea(\Phi(u_i,g_i)))}^2_{L^2(C_i,g_i)}+\norm{\tens_{g_i} u_i}_{L^2(C_i,g_i)}^2\big)\to 0, 
\eeq
$C=C(E_0, \al, c_{1,2})$. In particular, after passing to a subsequence 
 a harmonic bubble forms as described in part (ii) of Theorem \ref{thm:A}. 
 \item[(ii)] 
  If furthermore $\de_i\to 0$ fast enough so that 
 for some $q<\sqrt{2}$,
\beq 
\eps_{i,2}^{-1} \de^{1-\frac{1}{q}}\to 0,
\label{ass:relation}
\eeq
then the harmonic limit map $\hat u_\infty$  obtained by rescaling $u_i$ around the concentration point $\widehat{P}_*$ as in part (ii) of Theorem \ref{thm:A} 
is 
weakly conformal and hence a (possibly branched) minimal immersion spanning $\Gamma$.
Furthermore, if 
\beq 
\label{ass:new} 
(1+\abs{\log(\eps_{i,1}/\eps_{i,2})}) \cdot \eps_{i,2}^{-1} \de^{1-\frac{1}{q}}\to 0,
\eeq
then no Hopf-differential is lost near the boundary curve, namely 
\beq 
\label{claim:no-hopf-lost} 
\norm{\Phi(u_i,g_i)}_{L^1([-Y_i^-,-Y_i^-+\La]\times S^1,g_i)}\to 0\eeq for every $\La>0$.
\end{itemize}
\end{thm}

While the focus of the present paper is the analysis of boundary bubbles, it might be worth noting that the above result applies without change also for almost minimal maps $u_i$ satisfying a non-degenerate three-point condition.

We note that if the maps $u_i$ in the above theorem also satisfy a (not necessarily degenerating) Plateau-boundary condition on the second boundary curve $\{Y_i^+\}\times S^1$, then we may drop the assumption \eqref{ass:mv-different} in favour of imposing the separation condition \eqref{ass:no-linking}, see Section \ref{subsec:asympt} for details.

The paper is organised as follows: \\
First, in Section \ref{sec:proofs}, we provide the proofs of the results on almost minimal cylinders, Theorem \ref{thm:A} and Theorem \ref{thm:B}.
The proof of the first part of Theorem \ref{thm:A} is based on a combination of well-understood techniques from the analysis of bubbling for harmonic maps, developed in particular in \cite{bubbles-1}, \cite{bubbles-2}, \cite{bubbles-3} and \cite{Topping-bubbles}, with the Courant-Lebesgue Lemma and the regularity theory of \cite{D-S}, while the proof of the second part relies on the use of a compactness result for `almost-meromorphic functions', stated in Lemma \ref{lemma:key_holo} and proven in Section 3. 
A crucial part of the present work is to establish strong control of the Hopf-differential, both globally on the domain and in particular near the points where the three-point condition is imposed and where poles can form.
This control is used both to show that the formation of a boundary bubble forces the domain to degenerate (in the setting of Theorem \ref{thm:B} with rate \eqref{est:upper-ell-thm}), and also to show that  the Hopf-differential of the harmonic bubble is not just meromorphic but holomorphic, and hence zero, thereby establishing that the bubble is indeed a (branched) minimal immersion. To this end, we shall prove in Section \ref{sec:residues} that the residues at potential poles of almost minimal maps are controlled in terms of the tension and the horizontal part of the Hopf differential and use this result, which is stated in detail in Lemma \ref{lemma:key-res}, to complete the proof of Theorem \ref{thm:B} in Section \ref{subsect:proof-thmB}. Based on these general results on almost minimal cylinders, we may then derive the asymptotic convergence of the Teichm\"uller harmonic map flow \eqref{eq:fluss-metric}, \eqref{eq:fluss-map} in the final part \ref{subsec:asympt} of Section \ref{sec:proofs}. 

\textbf{Acknowledgement:}
The second author was supported by Engineering and Physical Sciences Research Council [EP/L015811/1].
\section{Proof of the main results}
\label{sec:proofs}
In this section, we assemble the proofs of all of our main results, with some of the arguments based on lemmas that are proven later on in Sections \ref{sec:compact} and \ref{sec:residues}.
To this end, we will first establish our main results concerning almost minimal maps from cylinders, stated in Theorems \ref{thm:A} and \ref{thm:B}, and later apply these results to derive the claimed asymptotic convergence of the flow in Section \ref{subsec:asympt}.

Before we begin with the proofs of our main theorems, we collect the results used in their proofs. 

To begin with, we recall that the stationarity condition implies that the trace of the Hopf-differential on the boundary curve is real, see Appendix \ref{subsect:app-stat} for details. 
As the horizontal space $H(g)$ can be equivalently characterised as the real part 
$H(g)=\Rea(\mathcal{H}(C_0,g))$ of the space $\mathcal{H}(C_0,g)$ of holomorphic quadratic differentials whose traces on $\partial C_0$ are real, the assumption of smallness of the horizontal part $P^H_g(\Rea(\Phi_i))$ of the Hopf-differentials in our main theorems can thus be thought of as $\Phi_i$ having only a small holomorphic part. 

Conversely, the antiholomorphic derivative of the Hopf-differential is controlled in terms of $\De_{g} u$, namely 
$\norm{\dbar \Phi(u,g)}_{L^1(C_0,g)}\leq C E(u,g)^{1/2}\cdot \norm{\De_g u}_{L^2(C_0,g)}$, 
and a key tool in the proof of Theorem \ref{thm:A} is the following compactness result for almost-meromorphic  Sobolev functions 
which is a generalisation of \cite[Lemma 2.3]{RT} and which is proven in Section \ref{sec:compact}. 

\begin{lemma}\label{lemma:key_holo}
Let $\phi_i\in L^1(\Om,\C)$, $\Om\subset \C $ open, be 
a sequence of functions for which 
\beq \label{ass:antihol}
\sup_i \norm{\dbar \phi_i}_{L^1(\Om)}+\norm{\phi_i}_{L^1(\Om)}<\infty.
\eeq
Suppose that 
there exist sets $S_i=\{p_i^j, j=1\ldots N\} \subset \Om$, for some $N\in\N$, 
so that $\phi_i\in W^{1,1}_{loc}(\Om\setminus S_i)$ and let $\Om'\subset\subset \Om$.
Then, after passing to a subsequence,
\beq
\label{claim:converg}
\tilde \phi_i:= \phi_i-\sum_j\res_{p_i^j}(\phi_i)\cdot  \frac{1}{z-p_i^j}\to \tilde \phi_\infty \text{ strongly in } L^1(\Om'),\eeq
for a limiting function $\tilde \phi_\infty$ which is furthermore holomorphic in the case that $\norm{\dbar \phi_i}_{L^1(\Om)}\to 0$. 
\end{lemma}
Here and in the following, 
the residues are defined by 
$
\res_{p_i^j}(\phi_i):=\frac{1}{2\pi \i}\lim_{r\to 0} \int_{\partial D_r(p_i^j)} \phi\, dz
$.

We use in particular the following consequence of the above lemma, whose proof is also provided in Section \ref{sec:compact}.

\begin{cor}
\label{cor:key-lemma}
Let $\phi_i:\Om= [0,\La)\times S^1\to \C$ be a sequence of functions
for which \eqref{ass:antihol} holds and suppose that  
$\phi_i\in W^{1,1}_{loc}(\Om\setminus S_i)$ 
 away from sets 
$S_i=\{p_i^j, j=1,2,3\}\subset \{0\}\times S^1$ and that the traces of $\phi_i$
 on $\{0\}\times S^1\setminus S_i$ are real.
\\
Suppose furthermore that the singular points converge to limit points
$p_i^j\to p_\infty^j$ and that 
$$\phi_i\to \phi_\infty \text{ in } L^1_{loc}(\Om\setminus S_\infty), \qquad S_{\infty}=\{p_\infty^j, j=1,2,3\}.$$
Then 
\beq \label{eq:squirrel}\psi_i:=\phi_i-\sum_j a_i^j h_{p_i^j}(z) \to \phi_\infty \text{ in } L^1_{loc}(\Om),\eeq
where 
\beq
\label{def:mod-fun}
h_{p}(z)=\frac{\i}{2\tanh(\half (z-p))}, \qquad z=s+\i\th \in \Om,\eeq
and where the numbers $a_i^j\in\R$ are chosen as follows:
\\
If $p_\infty^j$ does not coincide with any of the other limit points $p_\infty^l$ then $a_i^j=0$, while in the case that $k=\abs{\{l:p_\infty^l=p_\infty^j\}}\geq 2$, we set 
\beq a_i^j=\res_{p_i^j}(\phi_i)-\frac1k \sum_{l:p_\infty^l=p_\infty^j}  \res_{p_i^l}(\phi_i).
\eeq
\end{cor}

\begin{rmk}
\label{rmk:mod-no-horiz}
We note that since the functions $h_p$ by which we modify are such that 
$h_p(\overline{z-p})=-\overline{h_p(z-p)},$ we have that 
$\int_{S^1} \Rea(h_p(s+\i \th))\, d\th=0$
for every $s\in\R$.
 In particular, the corresponding quadratic differentials $h_p\cdot dz^2$ on a cylinder $([-Y^-,Y^+]\times S^1, \rho^2(ds^2+d\th^2))$ have no horizontal part, 
$$P_g^H(\Rea(h_p dz^2))=
 \frac{\Rea(dz^2)}{\norm{\Rea(dz^2)}_{L^2}^2}
\int_{-Y^-}^{Y^+}2\rho^{-2} \int_{S^1} 
\Rea(h_p(s+\i\th))\, d\th \, ds=0.
$$
\end{rmk}

The above compactness results will be crucial to gain control of the Hopf-differential near the concentration points $P_*^\pm$ where the bubbles form and where we have no chance of obtaining strong enough convergence of the map to be able to gain the necessary control on the Hopf-differential directly. They will also prove to be useful to analyse the Hopf-differential near points where the three-point condition is imposed, whether these points $p_i^j$ converge to a mutual limit point $P_*^\pm$ or  to distinct limits $P_\infty^{j,\pm}$.

Conversely, on half-discs $D_R^+(x_0)$ away from the points where the 
three-point condition is imposed, we can control the $H^2$-norm of maps using the following lemma which is a consequence of the regularity theory developed by Duzaar-Scheven in \cite{D-S}, see in particular \cite[Theorem 8.3]{D-S}.
\begin{lemma}\label{lemma:H2-estimate}[Corollary of \cite[Theorem 8.3]{D-S}]
Let $\Gamma\subset\mathbb{R}^n$ be a $C^3$ Jordan curve and let $\Gamma'\subset\Gamma$ be any sub-arc. Suppose $u$ is an element of the space $H^1_{\Gamma'}(D_R^+)$ of all $H^1$ functions on the half-disc $D_R^+=\{(x,y)\in D_R:y\geq0\}\subset \C$ for which $u\vert_{\{y=0\}\cap D_R}$ is a weakly monotone parametrisation of $\Gamma'$ and suppose that, for some $f\in L^2(D_R^+)$,
\beq\label{ineq:var-ineq}
\int_{D_R^+}\na u\cdot \na w\,dx + \int_{D_R^+}w\cdot f\,dx\geq 0 \text{ for all } w\in T_u^+H^1_{\Gamma'}(D_R^+).
\eeq
Then $u\in H^2(D^+_{R/2})$ and, for a constant $C>0$ that depends only on $\Gamma$, the modulus of continuity 
of $u$ on $D_{R}\cap \{y=0\}$, and an upper bound on the energy $E(u,g)$,
$$\int_{D^+_{R/2}}|\na^2u|^2\,dx\leq \frac{C}{R^2}\int_{D_{R}^+}|\na u|^2\,dx +C\int_{D_{R}^+}|f|^2\,dx.$$
\end{lemma}

We note that in contrast to \cite[Theorem 8.3]{D-S}, we do not need to 
assume smallness of energy in the above lemma since the maps we consider satisfy a \textit{linear} differential inequality of the form \eqref{ineq:var-ineq}
rather than a non-linear version of this as considered in \cite{D-S}, though in fact smallness of energy is in our case also just  
 a simple consequence of the fact that there are no non-constant harmonic maps from $S^2$ to $\R^n$, compare \cite[Lemma 3.8]{R-cyl}. 

We also recall the following well known consequence of the Courant-Lebesgue lemma (see e.g. \cite[Lemma 3.1.1]{Jost} or \cite[Lemma 4.4]{Struwe-book})
\begin{rmk} \label{rmk:Courant-Leb}
Given any $E_0>0$ and any $\eps>0$ 
we may choose $\de=\de(\eps,E_0,\Gamma^{\pm})\in (0,\half)$ so that for any map $u\in H^1_\Gamma(C)$ with $E(u,g_0)\leq E_0$ and any $p_0\in \{\pm Y^{\pm}\}\times S^1$ we have either 
$$\osc_{\{\pm Y^{\pm}\}\times S^1\cap B_\de^{g_0}(p_0)} <\eps \, \text{ or }\osc_{\{\pm Y^{\pm}\}\times S^1\cap (B_\de^{g_0}(p_0))^c} <\eps, 
$$
where here and in the following $g_0=ds^2+d\th^2$ denotes the flat metric on the cylinder. 
\\
In particular, if the three-point condition does not degenerate, then the traces $u_i\vert_{\partial C}$ of maps $u_i\in H^1_\Gamma(C)$ with uniformly bounded energy are equicontinuous. 
Conversely, if at least two of the three points $P_i^{j,\pm}$ at  which we impose the three-point condition converge to a common limit $P_*^\pm$ then the traces on $\{\pm Y^{\pm}\}\times S^1\setminus\{P_*^\pm\} $ converge to a constant locally uniformly.
\end{rmk}

\subsection{Proof of Theorem \ref{thm:A}} $ $\\
 Based on the tools collected above, we can now give the proof of our first main result. This proof is done in two main steps, first establishing that $\ell_i\to 0$ and then analysing the maps on the half-cylinders $C_i^\pm=C_i\cap \{\pm s>0\}$. This later part only uses properties of the maps on $C_i^\pm$ and is thus applicable also in more general situations, including the setting of Theorem \ref{thm:B}.

\begin{proof}[Proof of Theorem \ref{thm:A}]
Let $u_i:(C_i,g_i)\to \R^n$ be as in Theorem \ref{thm:A}. 
After passing to a subsequence, we can assume that the angular components $\th_i^{j,\pm}$ of the points $P_i^{j,\pm}$ at which the three-point condition is imposed converge and denote by $\th_\infty^{j,\pm}$ the resulting limits. 
We furthermore recall that since the boundary curves are disjoint and the energy of the maps is bounded uniformly, the length $L_i=Y_i^++Y_i^-$ of the cylinders is bounded away from zero by a uniform constant and hence the numbers $\ell_i$ are bounded from above, compare (3.10) in \cite{R-cyl}, so we may assume that $\ell_i\to \ell_\infty\geq 0$.

We first prove that a degeneration of (at least one of) the three-point conditions forces the cylinders $C_i$ to degenerate, i.e.~$\ell_\infty$ to be zero. 
Suppose that this is not the case. 
After passing to a subsequence, we may thus assume that $Y_i^\pm\to Y_\infty^\pm<\infty$ and hence $P_i^{j,\pm}\to P_\infty^{j,\pm}=(\pm Y_\infty^\pm, \th_\infty^{j,\pm})$. We set 
  $S_\infty:=\{P_\infty^{j,\pm}\}$ and $C_\infty=[-Y_\infty^-,Y^+_\infty]\times S^1$.
This convergence allows us to choose diffeomorphisms 
$f_i:C_\infty\to C_i$ so that $f_i(\pm Y_\infty^\pm\mp s,\th)=(\pm Y_i^\pm\mp s,\th)$ for all $s\in [0,\La]$ for some $\La>0$, i.e.~with $f$ conformal in
neighbourhoods $\Om^\pm=\{(s,\th): \abs{s\mp Y_\infty^\pm}\leq \La\}$  of the boundary curves, and so that the pulled back metrics converge
\beq
\label{conv:metric-proof-thmA-1}
f_i^*g_i\to g_\infty=\rho_{\ell_\infty}^2(ds^2+d\th^2) \text{ smoothly on }  C_\infty.\eeq
The uniform $H^2$-bounds on compact subsets of $C_\infty\setminus S_\infty$ from Lemma \ref{lemma:H2-estimate} allow us to pass to a subsequence so that on $C_\infty\setminus S_\infty$,
\beq 
\label{conv:map-proof-thmA-1} 
v_i:= u_i\circ f_i\to u_\infty \text{ weakly in } H^2_{loc} \text{ and strongly in } W^{1,q}_{loc} \text{ for every }  q<\infty.
\eeq
Viewing  $v_i$ as maps from the limiting surface $(C_\infty,g_\infty)$,
we hence know that the corresponding Hopf-differentials $\Phi(v_i,g_\infty)=\phi(v_i,g_\infty)dz^2$ converge in particular
$$\phi(v_i,g_\infty)\to \phi(u_\infty,g_\infty) \text{ strongly in } L^1_{loc}(C_\infty\setminus S_\infty).$$
Combining the uniform $H^2$-bounds on the maps with the convergence of the metrics \eqref{conv:metric-proof-thmA-1} and the 
assumption that $\norm{\Delta_{f_i^*g_i}v_i}_{L^2(C_\infty,f_i^*g_i)}=\norm{\Delta_{g_i}u_i}_{L^2(C_i,g_i)}\to 0$ implies that 
$$\norm{\Delta_{g_\infty}v_i}_{L^2(K,g_\infty)}\to 0 \text{ on compact sets } K \text{ of } C_\infty\setminus S_\infty$$
so that the obtained limit map is harmonic. 

In the neighbourhoods $\Om^\pm$ of the boundary where $f_i$ is conformal and indeed $f_i^*g_i=\beta_i^\pm(s)^2g_\infty$ for some $\beta_i^\pm\to 1$, we also have that 
$$\norm{\Delta_{g_\infty}v_i}_{L^2(\Om^\pm,g_\infty)}\leq \sup \beta_i^\pm\cdot \norm{\Delta_{f_i^*g_i}v_i}_{L^2(\Om^\pm,f_i^*g_i)}\leq C\norm{\Delta_{g_i}u_i}_{L^2(C_i,g_i)}\to 0$$
and hence in particular 
$$ \norm{\dbar \Phi(v_i,g_\infty)}_{L^1(\Om^\pm,g_\infty)}
\leq C\norm{\Delta_{g_\infty}v_i}_{L^2(\Om^\pm,g_\infty)}\to 0.
$$
We also recall that the Hopf-differential depends only on the conformal structure,
\beq 
\label{eq:Hopf-diff-agree} 
\Phi(v_i,g_\infty)=\Phi(u_i\circ f_i, f_i^*g_i)=f_i^* \Phi(u_i,g_i) \text{ on } \Om^\pm,\eeq
 so  that the stationarity condition implies that the traces of $\phi(v_i,g_\infty)$ on $\partial C_i$ are real, compare Appendix \ref{subsect:app-stat}. 
We may thus apply Corollary \ref{cor:key-lemma} to conclude that 
\beq 
\psi_i:=\phi(v_i,g_\infty)-\sum_{j,\pm} a_i^j h_{\pm Y_\infty^\pm+\i\th_i^{j,\pm}}\to
 \phi_\infty \text{ strongly in } L^1(C_\infty,g_\infty),
 \label{conv:mod-hopf-proof}\eeq
 where $a_i^j$ and $h_\cdot$ are as in that corollary.

We now first consider the case that the three-point condition degenerates on both of the boundary curves, i.e.~that both $\eps_\infty^\pm:=\lim_{i\to \infty}\eps_i^\pm=0$. In this situation, we obtain from the Courant-Lebesgue lemma that the traces of $v_i$ on both boundary curves converge to constant maps locally uniformly away from the two concentration points $P_*^\pm$, compare Remark \ref{rmk:Courant-Leb}.  
As the limit map $u_\infty$ is thus constant on both boundary curves and harmonic on the whole cylinder, it must be described by a parametrisation (proportional to arclength) $u_\infty(s,\th)=\al(s)$ of a straight line that connects two points $Q^\pm$ on the disjoint curves $\Gamma^\pm$ and that is hence non-trivial. 
As a consequence, the Hopf-differential of the limit map takes the form 
$c\cdot dz^2$ for some $c>0$, i.e.~$\Rea(\Phi(u_\infty,g_\infty))$ is a non-trivial element of $H(g_\infty)$ and has thus in particular non-zero projection onto $H(g_\infty)$. 

This implies that 
$$\norm{P^H_{g_\infty}(\Rea(\Phi(v_i,g_\infty)))}_{L^2(C_\infty,g_\infty)}\nrightarrow 0$$
as the projection of a quadratic differential $\Psi$ onto $H(g_\infty)$ is simply $\langle \Rea(\Psi),\Rea(dz^2)\rangle \frac{\Rea(dz^2)}{\norm{\Rea(dz^2)}_{L^2}^2}$, so combining the strong $L^1$-convergence of the modified Hopf-differentials $\psi_idz^2$ obtained from Corollary \ref{cor:key-lemma} with the fact that the 
modification we made has no horizontal part, compare 
Remark \ref{rmk:mod-no-horiz}, implies that 
\beqas
P^H_{g_\infty}(\Rea(\Phi(v_i,g_\infty)))=P^H_{g_\infty}(\Rea(\psi_i dz^2))
\to P^H_{g_\infty}(\Rea(\Phi(u_\infty,g_\infty)))\neq 0 .
\eeqas 
However, combining the fact that the metrics $g_\infty$ and $f_i^*g_i$ are conformal in a neighbourhood of the boundary, and the resulting relation \eqref{eq:Hopf-diff-agree}, with the convergence of the maps and metrics \eqref{conv:metric-proof-thmA-1}, \eqref{conv:map-proof-thmA-1} away from the boundary also implies that 
\beqs
\norm{P^H_{g_\infty}(\Rea(\Phi(v_i,g_\infty)))-P^H_{f_i^*g_i}(\Rea(\Phi(v_i,f_i^*g_i)))}_{L^2(C_\infty,g_\infty)}\to 0\eeqs
and hence that 
\beqas
\lim_{i\to \infty} \norm{P^H_{g_\infty}(\Rea(\Phi(v_i,g_\infty)))}_{L^2(C_\infty,g_\infty)}= \lim_{i\to \infty} \norm{P^H_{g_i}(\Rea(\Phi(u_i,g_i)))}_{L^2(C_i,g_i)}=0,\eeqas
leading to a contradiction in this first case where we assumed that $\ell_\infty>0$ and $\eps_\infty^+=\eps_\infty^-=0$.

Suppose now that still $\ell_\infty>0$ but that only one of the $\eps_\infty^\pm$ is zero, say 
$\eps_\infty^+>0=\eps_\infty^-$. 
Repeating the above argument, we obtain a limit $u_\infty:(C_\infty,g_\infty)\to \R^n$ which is harmonic, constant on $\{-Y_\infty^-\} \times S^1$, and
whose Hopf-differential is real on the boundary curves and holomorphic away from the points $P_\infty^{j,\pm}$ where it might a priori have poles of order one. 
However, since $\eps_\infty^+>0$, the assumptions of the theorem ensure that 
the projection of the Hopf-differential onto the space $\Var^+(g_i)$, corresponding to changes of $(b^+,\phi^+)$, tends to zero as well. By  \cite[Lemma 3.3]{R-cyl}, compare also \cite[Proof of Theorem 2.7(i)]{R-cyl}, this implies that 
 the 
stationarity condition holds for the limit $u_\infty$ for all vector fields supported in a neighbourhood of the boundary curve $\{Y_\infty^+\}\times S^1$ 
rather than just for such vector fields that satisfy the constraint $Y(P_\infty^{j,+})=0$. 
By Remark \ref{lemma:no-poles}, this excludes the possibility that $\Phi(u_\infty,g_\infty)$ has poles at points on the boundary curve $\{Y_\infty^+\}\times S^1$. 
On the other hand, 
 as $u_\infty$ is constant on $\{-Y_\infty^-\}\times S^1$, 
 standard regularity theory yields that $u_\infty$ is smooth
in a neighbourhood of the other boundary curve $\{-Y_\infty^-\}\times S^1$
so that $\Phi(u_\infty,g_\infty)$ may in particular not form any poles there either. 
As the stationarity condition ensures that $\Phi(u_\infty,g_\infty)$ is real on both boundary curves, $\Rea(\Phi(u_\infty,g_\infty))$ is hence again an element of 
 $H(g_\infty)$, and thus, by the argument from above, must be zero. 

This implies that the map $u_\infty:C_\infty\to\mathbb{R}^n$ is harmonic, conformal, and constant on one of the boundary curves and hence has vanishing normal derivative there, and so may be reflected across that curve to give a map from $(-Y_\infty^--L_\infty,Y^+)\times S^1$ which, by \cite{G-O-R}, is either constant or a minimal immersion away from finitely many points. As $\na u_\infty=0$ on $\{-Y_\infty^-\}\times S^1$, we would thus need $u_\infty$ to be constant which contradicts the fact that 
$u_\infty\vert_{\{Y_\infty^+\}\times S^1}$ parametrises $\Gamma^+$. 

Having thus established that  
$\ell_i\to 0$
as claimed in the theorem and hence that $Y_i^\pm\to \infty$, we now turn to the analysis of the shifted maps
$u_i^\pm=u_i\circ f_i^\pm$, $f_i^\pm(s,\th)=(\pm Y_i^\pm \mp s,\th)$, which are defined on larger and larger sub-cylinders 
$[0,Y_i^\pm)\times S^1$ of $C_\infty=[0,\infty)\times S^1$. 
As these maps and the following analysis use only the information on the corresponding half-cylinders $C_i^\pm=\{\pm s>0\}\cap C_i$ and the fact that $Y_i^\pm\to\infty$, 
we remark that the following argument is applicable not only in the setting of Theorem \ref{thm:A} but also in the setting of Theorem \ref{thm:B} (there once we have established that $\ell_i\to 0$).

Repeating the argument from the first part of the proof, now applied to $u_i^\pm$ instead of $v_i$, we conclude that, after passing to a subsequence,
the maps converge away from $S_\infty^\pm:=\{(0,\th_\infty^{j,\pm})\}$,
$$u_i^\pm\to u^\pm_\infty \text{ weakly in }H^2_{loc}(C_\infty\setminus S_\infty^\pm) \text{ and strongly in } W^{1,q}_{loc}(C_\infty\setminus S_\infty^\pm), \,q<\infty, 
$$ to a limit map $u^\pm_\infty$ which is harmonic, has finite energy and which may thus be extended across the puncture to a harmonic map defined over the disc using the removability of singularities theorem of Sacks-Uhlenbeck \cite{S-U}.

The obtained maps $u_\infty^\pm:C_\infty\to \R^n$ have the following properties:
if $\eps_\infty^\pm=0$ then, by Remark \ref{rmk:Courant-Leb}, the traces of $u_i^\pm$ converge locally uniformly to a constant $q^\pm\in \Gamma^\pm$ on $\{\pm Y_\infty^\pm\}\times S^1\setminus \{P_*^\pm\}$, so the limit map $u_\infty^\pm$ is constant on the boundary of $C_\infty$ and hence on the whole cylinder. 
Conversely, if one of the $\eps_\infty^\pm>0$, then the stationarity condition \eqref{eq:stationar} is satisfied for all 
vectorfields $X$ with support in a neighbourhood of the corresponding boundary curve $\{\pm Y_\infty^\pm\} \times S^1$ and thus Remark \ref{lemma:no-poles} ensures that $\Phi(u_\infty^\pm,g_\infty)$ is real and 
has no poles on $\{0\}\times S^1$. 
Extending $\Phi(u_\infty^\pm,g_\infty)$ by reflection to the whole of $(-\infty,\infty)\times S^1$ gives a quadratic differential which must be represented by a holomorphic function on $(-\infty,\infty)\times S^1$  with finite $L^1$-norm, i.e. zero, so we must have that have that $\Phi(u_\infty^\pm,g_\infty)\equiv 0$.  The limit map $u_\infty^\pm$ is thus not only harmonic but also conformal, and hence a (possibly branched) minimal immersion. This completes the proof of part (i) of the theorem.

Before we analyse the bubbles, we note that since the diffeomorphisms $f_i^\pm$ are conformal on the whole cylinder and so, in particular, in a neighbourhood of the boundary, the analysis of the Hopf-differential carried out in the first part of the proof still applies. Thus, after modifying the Hopf-differentials as in Corollary \ref{cor:key-lemma}, we obtain strong local $L^1$-convergence to the Hopf-differential of these limit maps $u_\infty^\pm$, i.e.~to zero;
 see Corollary \ref{cor:chocolate} below for a precise statement. 

We now analyse the bubbles forming at the concentration point
 $P_*^-$ (if $\eps_\infty^-=0$), with the same argument of course applying also for $u_i^+$ in the case that also $\eps_i^+\to 0$.
We first pull back the maps $u_i^-:[0,Y_i^-)\times S^1\to \R^n$
 by the 
conformal diffeomorphism
$re^{\i \th}\mapsto (s=-\log(r), \th)$ to annuli $D\setminus D_{r_i}$ in the punctured unit disc and further pull-back the resulting maps by the unique M\"obius transform $M_{b_i^-,\phi_i^-}:D\to D$ which maps $e^{\i \frac{2\pi j}{3}}$ to $e^{\i\th_i^{j,-}}$ so that, overall, the points $P_{i}^{j,-}$ at which we imposed the three-point condition are pulled-back to $e^{\i \frac{2\pi j}{3}}$. 
 As our subsequence was chosen so that the angles $\th_i^{j,-}$ converge, we know that $b_i^-\to b^-_\infty\in \partial D$ and $\phi_i^-\to \phi^-_\infty$. Hence  
the resulting conformal diffeomorphisms
$\hat f^-_i:\Om^-_i\subset D\to  [0,Y^-_i)\times S^1$
 converge to a constant map away from the point $\widehat P_*^-=-b_\infty^-\in \partial D$.
The domains $\Om_i^-$, on which these diffeomorphisms, and hence the pulled back maps $\hat u_i ^-=u_i^-\circ \hat f^-_i$,  
are defined, exhaust $D\setminus \{\widehat P_*^{-}\}$. As all involved diffeomorphisms are conformal, we note that the pulled-back metrics are given by
$\hat\beta_i^2 g_{D}$ for conformal factors $\hat\beta_i$ that converge to zero away from $\widehat P_*^-$.

The rescaled maps $\hat u_i^-$ are thus almost harmonic also with respect to the standard metric on the disc in the sense that 
\beqs\norm{\Delta_{g_{D}}\hat u_i^-}_{L^2(K,g_{D})}\leq \sup_K \hat\beta_i \norm{\Delta_{g_i}u_i}_{L^2(C_i,g_i)}
\to 0 \text{ for every } K\subset\subset D\setminus \{\widehat P_i^-\},\eeqs so, using Lemma \ref{lemma:H2-estimate} and the Courant-Lebesgue lemma, we can conclude that a subsequence converges to a harmonic map $\hat u_\infty^-$ which spans $\Gamma^-$, where the obtained convergence is uniform on the boundary and the usual local weak $H^2$ and strong $W^{1,q}$ ($q<\infty$)  convergence away from the concentration point $\widehat P_*^-
$ and away from the three points $e^{\i \frac{2\pi j}{3}}$ at which the three-point condition is imposed. We furthermore remark that \cite[Lemma 3.11] {R-cyl} excludes the possibility that energy concentrates at the three points $e^{\i \frac{2\pi j}{3}}$ so that the maps also converge  strongly in $H^1_{loc}(D\setminus\{\widehat P_*^-\})$.
Since the stationarity condition is invariant under conformal changes of the metric, we thus obtain that 
$\int \Rea(\Phi(u_\infty,g_\infty))\cdot L_Xg \,dv_g=0 $
holds true for every vector field $X$ that is supported in $D\setminus \{\widehat P_*^-\}$ and that furthermore satisfies $X(e^{\i\frac{2\pi j}{3}})=0$, which completes the proof of the theorem.
\end{proof}

It is important to observe at this point that the above proof yields that, for $i$ large, the Hopf-differential on $C_i^-$ is essentially described by the modification $\sum_j a_i^{j,-}h_{P_i^{j,-}} dz^2$ obtained in  Corollary \ref{cor:key-lemma}, which in turn is determined only by the residues of the Hopf-differential at the three points $P_i^{j,-}$. As observed above, this argument is applicable not only in the setting of the result that we have just proven, but once we prove that the assumptions of Theorem \ref{thm:B} imply that $\ell_i\to 0$ also in that situation, so it is important to record that we have shown in particular:

\begin{cor}\label{cor:chocolate}
Let $u_i$ be a sequence of maps 
from  
cylinders $C_i$ (as in \eqref{def:cyl} and \eqref{ass:cyl})
with uniformly bounded energy which are so that $u_i\vert_{\{-Y_i^-\}\times S^1}$ are weakly monotone parametrisations of  
 a $C^3$ Jordan curve $\Gamma$ and so that the $u_i$ satisfy 
the stationarity condition \eqref{eq:stationar} at $\{-Y_i^-\}\times S^1$ except at three points $P_i^j=(-Y_i^-,\th_i^{j,-})$ for which furthermore $u_i(P_i^j)=Q^j$,  $Q^{1,2,3} \in \Gamma$ distinct points, and $\th_i^{j,-}\to \th_{\infty}^{j,-}$.
\\
Suppose that these maps are almost harmonic, i.e.~that $\norm{\De_{g_i} u_i}_{L^2(C_i,g_i)}\to 0$, and that $\ell_i\to 0$. Let $\phi_i=\phi(u_i,g_i)$ be the functions describing the Hopf-differentials and let  
$m_i^-:=\sum_j a_i^{j,-}h_{P_i^j}$ be defined as in Corollary \ref{cor:chocolate}, i.e.~in the case that all limits $\th_\infty^{j,-}$ are distinct we have $m_i^-\equiv 0$, while 
if $k\in \{2,3\}$ of these limits agree, say $\th_{\infty}^{1,-}=\ldots =\th_{\infty}^{k,-}$, then 
\beq \label{eq:mod-fn}
m_i^-=\sum_{1\leq j\leq k}
\res_{P_i^{j,-}}(\phi_i) \cdot h_{P_i^{j,-}}-\bigg(\frac{1}{k}\sum_{1\leq j\leq k}\res_{P_i^{j,-}}(\phi_i)\bigg) \cdot  \sum_{1\leq j\leq k}h_{P_i^{j,-}}.\eeq
Then the Hopf-differentials of $u_i$ are essentially described by the 
corresponding quadratic differentials $M_i^-=m_i^- dz^2$
in the sense that, for every $\La<\infty$,
\beq 
\label{conv:Hopf-mod-to-zero}
\norm{\Phi(u_i,g_i)-M_i^-}_{L^1(K_i^{-,\La},g_i)}\to 0 \text{ as } i\to \infty\eeq 
on `chunks' $K_i^{-,\La}:=\{(s,\th): 0\leq Y_i^{-}+ s\leq \La\}$ of 
size $\La$ around $\{-Y_i^-\}\times S^1$. 
\end{cor}

\subsection{Proof of Theorem \ref{thm:B}}
\label{subsect:proof-thmB}
$ $\\
The analysis carried out in the previous section implies that a lack of control on the residues of the Hopf-differential at the points where the three-point condition is imposed is 
the main obstruction to the obtained bubbles being not only harmonic but indeed minimal. We will obtain the required control on these residues as a consequence of Lemma \ref{lemma:key-res} that we state below and that will be proven in Section \ref{sec:residues}.

To state this lemma we 
consider as usual  maps from 
 a cylinder of the form
\beq
\label{def:cyl+-}
(C,g):=\big([-Y^-,Y^+]\times S^1,\, \rho^2_\ell(s)(ds^2+d\th^2)\big), \text{ where }Y^\pm=\tfrac{2\pi}{\ell}(\tfrac{\pi}2 -c^\pm \ell)\geq 1
\eeq
 for some $c^\pm\in [c_1,c_2]$, $\ell\in (0,\ell_0)$, where $\ell_0$, $c_1$ and $c_2$ are arbitrary fixed, positive constants.

\begin{lemma}
\label{lemma:key-res}
For any $\al, E_0, c_{1,2}>0$ there exists $\bar\ell=\bar\ell (\al,E_0, c_{1,2})>0$, determined in Lemma \ref{lemma:ell-upper} below,  so that the following holds.
\\
Let $u\in H^1(C,\R^n)$ be a map on a cylinder as above for which $\ell\leq \bar \ell$ with $E(u,g)\leq E_0$ and $\De_g u\in L^2(C,g)$. Suppose that $u$ satisfies the stationarity condition \eqref{eq:stationar} on 
 $\{-Y^-\}\times S^1$ except at three points $P^{1,2,3}=(-Y^-,\th^j)$ 
 which are ordered so that
\beqa
\label{def:epsilon-wants-cookies}
\th^1=\th^2-\eps_1<\th^2<\th^3=\th^2+\eps_2 \text{ for some } 0<\eps_1\leq \eps_2\leq 2\pi-(\th^3-\th^1),
\eeqa 
that the Hopf-differential $\Phi=\phi dz^2$ of $u$ is in $W^{1,1}_{loc}(C\setminus\{P^{1,2,3}\})$ and that the mean values $M_{u}(\pm Y^{\pm}):=\fint_{\{\pm Y^{\pm}\}\times S^1} u \, d\th$ are so that
\beq
\label{ass:MV-lower}
0<\alpha\leq \abs{M_{u}(-Y^{-})-M_{u}(Y^{+})}.
\eeq
Then the residues of the Hopf-differential  at the points $P^j$ satisfy the following estimates:
\beq
\label{claim:sum-res}
\abs{\res_{P^1}(\phi)+\res_{P^2}(\phi)+\res_{P^3}(\phi)}\leq C\norm{\tens_g u}_{L^2(C,g)}+C \exp(-c\norm{P_g^H(\Rea(\Phi))}_{L^2(C,g)}^{-2}),
\eeq
while for every $q<\sqrt{2}$  
\beq 
\label{claim:res} 
\eps_1\eps_2 (\abs{\res_{P^1}(\phi)}+\abs{\res_{P^2}(\phi)}) +\eps_2^2 \abs{\res_{P^3}(\phi)} \leq C_q\big[\norm{\tens_g u}_{L^2(C,g)}^{1-\frac{1}{q}}+\norm{P^H_g(\Rea(\Phi))}_{L^2(C,g)}^{4(1-\frac{1}{q})}\big]
\eeq
and 
\beq 
\label{claim:eps-sum-res} 
\abs{\eps_2\res_{P^3}(\phi) -\eps_1\res_{P^1}(\phi)}\leq C_q \big[\norm{\tens_g u}_{L^2(C,g)}^{1-\frac{1}{q}}+\norm{P^H_g(\Rea(\Phi))}_{L^2(C,g)}^{4(1-\frac{1}{q})}\big]
\eeq
for constants $C,C_q>0$ that in addition to $\al, E_0$ and $c_{1,2}$ are allowed to depend on upper bounds on $\norm{\De_g u}_{L^2}$ and $\abs{M_u(-Y^-)-M_u(Y^+)}$ and for $C_q$ additionally on the choice of $q<\sqrt{2}$.
\end{lemma}

In addition, we shall need that the length of the central geodesic of such a cylinder is controlled by the following lemma, which establishes in particular the claimed rate \eqref{est:upper-ell-thm} of the degeneration in Theorem \ref{thm:B}.

\begin{lemma}
\label{lemma:ell-upper}
Let $(C,g)$ be a cylinder as described in \eqref{def:cyl+-}, let $u\in H^1(C,\R^n)$ be a map for which \eqref{ass:MV-lower} is satisfied for some $\alpha>0$ and let $\Phi$ be the corresponding Hopf-differential. Then 
 \beqa\label{est:hor-lower}
 \norm{P^H_g(\Rea(\Phi))}_{L^2(C,g)}\geq C \alpha^2 \ell^{1/2}-C\ell^{3/2}-C\ell^{-1/2}\norm{\tens_g u}_{L^2(C,g)}^2
 \eeqa
for a constant $C$ that depends only on an upper bound $E_0$ on the energy of $u$ and the constants $c_{1,2}$ and $\ell_0$ in the definition of the cylinder.
\\
In particular, there exists a number $\bar\ell=\bar\ell(\al,E_0, c_{1,2})>0$
so that if the length $\ell$ of the central geodesic of the cylinder $C$ is no more than $\ell\leq \bar\ell$, then indeed
\beq\label{est:upper-ell}
\ell\leq C\big(\norm{P_g^H(\Rea(\Phi))}^2_{L^2(C,g)}+\norm{\tens_g u}_{L^2(C,g)}^2\big),
\eeq
where $C$ now additionally depends on $\alpha$.
\end{lemma}
In the special case of maps which are both harmonic and conformal, the above lemma of course just reduces to the well known fact that the conformal structures of solutions to the Douglas-Plateau problem are constrained in terms of the energy and $\alpha$. The above lemma should however \textit{not} be seen as a generalisation of this fact to maps which are almost harmonic and almost conformal as control on the horizontal part of the Hopf-differential and on $\De_{g} u$ is insufficient to control the full Hopf-differential $\Phi$ due to the poles of $\Phi$ in the points where the three-point condition is imposed.

Instead, this lemma will be a key ingredient in proving that almost harmonic maps whose Hopf-differentials have small horizontal part are indeed almost conformal in situations where a boundary bubble forms. 
While the proof of this lemma is given in Section \ref{sec:residues}, we first use it to give the proof of our main result on almost minimal maps from cylinders.
\begin{proof}[Proof of Theorem \ref{thm:B}]
Let $u_i:C_i\to \R^n$ be as in Theorem \ref{thm:B}. We first note 
that since the metrics $g_i$ are uniformly equivalent to the flat metric $g_0$ near the boundary of the cylinders $C_i$, the assumption \eqref{ass:relation}, respectively \eqref{ass:new}, involving the distances $\dist_{g_i}(P_i^j,P_i^k)$ of the points $P_i^j=(-Y_i^-,\th_i^j)$ at which the three-point condition is imposed is 
 equivalent to asking that \eqref{ass:relation}, respectively \eqref{ass:new}, is satisfied instead for the differences between the corresponding angles $\th_i^j$. Following the notation from Lemma \ref{lemma:key-res}, we can thus consider instead $\eps_{i,j}$ defined as in \eqref{def:epsilon-wants-cookies}, i.e.~after possibly reordering the points, assume that 
$$\eps_{i,1}:=\th_i^2-\th_i^1\leq \eps_{i,2}:=\th_i^3-\th_i^2\leq 2\pi-(\th_i^3-\th_i^1).$$
As the lengths $\ell_i$ of the central geodesics of the cylinders $(C_i,g_i)$ are assumed to be no more than the constant $\bar \ell$ from the above Lemma \ref{lemma:ell-upper}, we first note that this lemma implies that $\ell_i\to 0$ with the rate given in \eqref{est:upper-ell-thm}. As such we may apply the second part of the proof of Theorem \ref{thm:A} to conclude that, after passing to a subsequence, the  translated maps
$u_i^{-}:[0,Y^-_i)\times S^1\to \R^n$ form a bubble at a concentration point $P_*^-$ as described in detail in Theorem \ref{thm:A}.

As the rescaled maps  $\hat u_i^-=u_i^-\circ \hat f_i^-$ defined in Theorem \ref{thm:A} converge strongly in $H^1$ to the harmonic bubble $\hat u_\infty^-$ away from $\widehat P_*^-\in \partial D$, we note that $\hat u_\infty^-$ is conformal if and only if 
\beq
\label{frog}
\norm{\Phi(\hat u_i^-,g_{D})}_{L^1(K,g_D)}\to 0 \text{ for all } K\subset\subset D\setminus \{\widehat P_*^-\}.\eeq
Because the diffeomorphisms $\hat f^-_i: \Om^-_i\subset D\setminus\{\widehat P_*^-\}\to [0,Y_i^-)\times S^1$ are conformal and the $L^1$-norm of quadratic differentials is conformally equivalent, \eqref{frog}  is equivalent to 
\beq \label{equiv-frog}
\norm{\Phi(u_i^-,g_0)}_{L^1(\hat f_i^-(K),g_0)}\to 0 \text{ for all } K\subset\subset D\setminus \{\widehat P_*^-\},
\eeq
where $g_0=ds^2+d\th^2$. 
As we shall see below, this in turn will follow provided we show that  
\beq \label{equiv-frog-2}
\norm{\Phi(u_i,g_i)}_{L^1(D_{M\eps_{i,1}}(P_i^2),g_0)}\to 0 \text{ for every  } M<\infty,
\eeq
and we shall prove that this is the case if the tension and horizontal part of Hopf-differential decay according to \eqref{ass:relation}. Moreover, we shall see that the slightly stronger assumption \eqref{ass:new} ensures that no Hopf-differential at all can be lost on any finite length chunk around the boundary curve, namely that 
 \beq \label{conv:Hopf-to-zero} 
\norm{\Phi(u_i,g_i)}_{L^1([-Y_i^-,-Y_i^-+\La]\times S^1,g_i)}\to 0 \text{ for every } \La>0.
\eeq
To prove \eqref{equiv-frog-2} and \eqref{conv:Hopf-to-zero}, we  note that, by Corollary \ref{cor:chocolate}, the claim 
\eqref{equiv-frog-2} is equivalent  
 to \beq \label{conv:mod-fn-to-zero-frog} 
\norm{M_i^-}_{L^1(D_{M\eps_{i,1}}(P_i^2), g_0)} \to 0
\text{ for every } M>0 \eeq
for the quadratic differentials $M_i^-=m_i^-dz^2$ that are obtained by modifying the Hopf-differentials $\Phi(u_i,g_i)=\phi_i dz^2$ as described in  Corollary \ref{cor:chocolate}, while  
 \eqref{conv:Hopf-to-zero}
is equivalent to proving that
 \beq \label{conv:mod-fn-to-zero} 
\norm{M_i^-}_{L^1(K_i^{-,\La},g_0)}\to 0 \text{ on }  K_i^{-,\La}=
\{(s,\th): 0\leq Y_i^-+s\leq \La\}. \eeq
In the case that $\eps_{i,2}$ is bounded away from zero uniformly, Lemma \ref{lemma:key-res} 
yields, for $q<\sqrt{2}$, 
$$\abs{\res_{P_i^{1,2}}(\phi_i)}\leq C\eps_{i,1}^{-1}\cdot \de_i^{1-\frac{1}{q}}, \qquad \de_i \text{ defined by \eqref{ass:tension}}, $$
allowing us to bound the coefficients $a_i^{1}=-a_i^2=\half(\res_{P_i^1}(\phi_i)- \res_{P_i^2}(\phi_i))$
of the function 
$m_i^-=a_i^1(h_{P_i^1}-h_{P_i^2})$ by which we modify the Hopf-differential. 
To estimate $\norm{M_i^-}_{L^1}=2\norm{m_i^-}_{L^1}$ over either $V_i:=D_{M\eps_{i,1}}(P_i^2)\subset C_i^-$, or over $V_i:=K_{i}^{-,\La}$, we consider the shifted functions $m_i^-(\cdot -P_i^2)=a_i^1(h_{-\i\eps_{i,1}}-h_0)$ on the corresponding domains $\tilde V_i=D_{M\eps_{i,1}}(0)$, respectively $\tilde V_i=[0,\La]\times S^1$ and split these domains into the disc $D_{4\eps_{i,1}}=D_{4\eps_{i,1}}(0)$ 
and the rest of $\tilde V_i$ resulting in 
 \beqas
\norm{M_i^-}_{L^1(V_i)}&
\leq 4 \abs{a_i^1} \norm{h_0}_{L^1(D_{5\eps_{i,1}})}
+C\abs{a_i^1}  \cdot \norm{h_{0}-h_{-i\eps_{i,1}}}_{L^1(V_i\setminus D_{4\eps_{i,1}})}\\
&\leq C \de_i^{1-\frac1q} (1+\norm{z^{-2}}_{L^1(V_i\setminus D_{4\eps_{i,1}})}),
\eeqas
where the second step follows using Taylor and the fact that 
$\norm{h_0}_{L^1(D_r)}\leq Cr$ for $0<r<1$.

If \eqref{ass:relation} is satisfied we hence obtain that  
$$\norm{M_i^-}_{L^1(D_{M\eps_{i,1}}(P_i^2))}\leq C \de_{i}^{1-\frac1q}\to 0$$
while under the stronger assumption \eqref{ass:new}
 we indeed obtain that 
 $$\norm{M_i^-}_{L^1(K_{i}^{-,\La})}\leq C(\abs{\log(\eps_{i,1})}+1) \de_{i}^{1-\frac1q}\to 0.$$
 This completes the proof of the theorem in the case that the $\eps_{i,2}$ are bounded away from zero.

To show that \eqref{conv:mod-fn-to-zero} holds true also in the case where all three points $P_i^j$ converge to the same limit point, i.e.~that both $\eps_{i,1}$ and $\eps_{i,2}$ tend to zero, we note that in this situation, the function $m_i^-$ from   Corollary \ref{cor:chocolate} is characterised by  
\beqas
m_i^-(\cdot -P_i^2)=&\res_{P_i^1}(\phi_i)(h_{-\i\eps_{i,1}}-h_0)+\res_{P^3_i}(\phi_i)(h_{\i\eps_{i,2}}-h_0)\\
&-\tfrac{1}{3}\big(\sum\res_{P^j_i}(\phi_i)\big) (h_{-\i\eps_{i,1}}-2h_0+h_{\i\eps_{i,2}}).
\eeqas
To estimate the norm of this function over $V_i$ as above, 
 we now split up the domain into the disc of radius $4\eps_{i,1}$, the annulus $D_{4\eps_{i,2}}\setminus D_{4\eps_{i,1}}$ and the remaining cylinder to get 
\beqas
\norm{M_i^-}_{L^1(V_i)}&
\leq C\norm{\res_{P_i^1}(\phi_i)(h_{-\i \eps_{i,1}}-h_{0}) +\res_{P_i^3}(\phi_i)
(h_{\i\eps_{i,2}}-h_0)}_{L^1(\tilde V_i)}\\
&\quad+
C\abs{\sum\res_{P_i^j}(\phi_i)}\\
&\leq 
C\abs{\res_{P_i^1}(\phi_i)} \cdot \norm{h_{0}}_{L^1(D_{5\eps_{i,1}})}
+C\abs{\res_{P_i^3}(\phi_i)} \cdot \norm{h_{0}}_{L^1(\tilde V_i\cap D_{5\eps_{i,2}})}\\
& \quad 
+ C\abs{\res_{P_i^1}(\phi_i)}\cdot \norm{h_{-\i \eps_{i,1}}-h_{0}}_{L^1(\tilde V_i\cap (D_{4\eps_{i,2}}\setminus D_{4\eps_{i,1}}))}\\
&\quad 
+C\norm{\res_{P_i^1}(\phi_i)(h_{-\i \eps_{i,1}}-h_{0})
+\res_{P_i^3}(\phi_i)
(h_{\i\eps_{i,2}}-h_0)}_{L^1(\tilde V_i\setminus D_{4\eps_{i,2}})}\\
&\quad+C\abs{\sum\res_{P_i^j}(\phi_i)}.
\eeqas
The last term in this estimate is controlled by \eqref{claim:sum-res} which, together with 
 \eqref{ass:tension}, implies that 
 $$C\abs{\sum\res_{P_i^j}(\phi_i)} \leq C\de_i+C\exp(-c\delta_i^{-1/2})\to 0.$$
 Similarly,  we may use \eqref{claim:res} to bound the first two terms in the above estimate by 
\beqas
C\eps_{i,1} \abs{\res_{P_i^1}(\phi_i)}+C\eps_{i,2}\abs{\res_{P_i^3}(\phi_i)}
\leq C \eps_{i,2}^{-1}
  \de_i^{1-\frac{1}{q}}\to 0\eeqas
  if the weaker assumption \eqref{ass:relation} of the theorem is satisfied. 
  
Writing these three terms for short as $o(1)$
and applying Taylor's theorem to $h_{-\i\eps_{i,1}}-h_0$ and $h_{\i\eps_{i,2}}-h_0$,
we may thus apply \eqref{claim:res} and \eqref{claim:eps-sum-res} to obtain that 
\beqas 
\norm{M_i^-}_{L^1(V_i)}
\leq &\, 
C\abs{\res_{P_i^1}(\phi_i)}\cdot \eps_{i,1}\norm{z^{-2}}_{L^1(\tilde V_i\cap (D_{4\eps_{i,2}}\setminus D_{4\eps_{i,1}}))}
 \\
&
+C\abs{-\res_{P_i^1}(\phi_i)\cdot \eps_{i,1}+\res_{P_i^3}(\phi_i)\cdot \eps_{i,2}}
\cdot \norm{z^{-2}}_{L^1(\tilde V_i\setminus D_{4\eps_{i,2}})}\\
&
+C\big[ \abs{\res_{P_i^1}(\phi_i)}\cdot \eps_{i,1}^2+\abs{\res_{P_i^3}(\phi_i)} \cdot \eps_{i,2}^2\big]
 \cdot \norm{z^{-3}}_{L^1(\tilde V_i \setminus D_{4\eps_{i,2}})}+o(1)
 \\
 \leq& \, C\big[  \eps_{i,2}^{-1}
\norm{z^{-2}}_{L^1(\tilde V_i\cap (D_{4\eps_{i,2}}\setminus D_{4\eps_{i,1}}))}
  + 
 \abs{\log(\eps_{i,2})}
  + \eps_{i,2}^{-1} \big] \de_i^{1-\frac{1}{q}}+o(1)
\eeqas
In the case that $V_i=D_{M\eps_{i,1}}(P_i^2)$ we take the norm of $z^{-2}$ only over annuli $D_{M\eps_{i,1}}\setminus D_{4\eps_{i,1}} $ over which this function has bounded $L^1$ norm so that \eqref{ass:relation} again ensures that 
$$\norm{M_i^-}_{L^1(D_{M\eps_{i,1}})}\leq C\eps_{i,2}^{-1} \de_i^{1-\frac1q}\to 0.$$
Conversely, the stronger assumption \eqref{ass:new} allows us to deal with the additional $\log(\eps_{i,2}/\eps_{i,1})$ term arising if $V_i= K_i^{-,\La}$ so that in this case  \eqref{conv:mod-fn-to-zero} follows. 

Finally, to explain why \eqref{equiv-frog} follows from  \eqref{equiv-frog-2}, we recall that on $\partial D$ the angular component of $\hat f_i^-:D\to [0,Y_i^-)\times S^1$ is described by the M\"obius transform 
$M_{b_i,\phi_i}$ for which
$M_{b_i,\phi_i}(e^{\frac{2\pi j}{3}\i})=\th_i^{j}$. With angles ordered as in \eqref{def:epsilon-wants-cookies}, we can easily check that
$$\Arg(-b_i)\in [0,\tfrac{\pi}{3}] \text{ and } 1-\abs{b_i}\in [c\cdot \eps_{i,1},C\cdot \eps_{i,1}]  $$
for some universal constants $c,C>0$. As $\abs{M_{b,\phi}'(z)}=\frac{1-\abs{b}^2}{\abs{1+\bar b z}^2}\leq C_r (1-\abs{b})$ on $D\setminus D_r(b)$, we thus find that the image of $K\subset\subset D\setminus\{\widehat P_*^-=-b_\infty\} $ under $M_{b_i,\phi_i}$ is contained in small discs of the form 
$D_{\tilde M_K \eps_{i,1}}(e^{\i\th_i^{2}})\subset D$ and hence also $\hat f_i^-(K)\subset D_{M_{K} \eps_{i,1}}(P_i^{2})$ for $i$ sufficiently large.
\end{proof}

\subsection{Asymptotic analysis of the flow}
\label{subsec:asympt}$ $\\
We finally turn to the analysis of the flow and the proof of Theorem \ref{thm:1}.
So let $(u,g)_{t\in[0,\infty)}$ be a solution of the flow \eqref{eq:fluss-metric}, \eqref{eq:fluss-map} as obtained in \cite{R-cyl} which satisfies both the stationarity condition \eqref{eq:stationar} at the boundary except at $P_0^{j,\pm}(t)=(\pm Y_{\ell(t)},\frac{2\pi j}{3})$ and the energy inequality \eqref{eq:energy-cond}, which implies in particular that 
\beq
\label{est:tension}
\int_0^\infty \norm{\De_g u}_{L^2(C_0,g)}^2\, dt+
\int_0^\infty \norm{P_g^{H}(\Rea(\Phi(u,g)))}_{L^2(C_0,g)}^2\,dt<\infty\eeq
and 
\beqs
\int_0^\infty \eta(\abs{b^+})^2\norm{P_g^{\Var^+}(\Rea(\Phi(u,g)))}_{L^2(C_0,g)}^2+\eta(\abs{b^-})^2
\norm{P_g^{\Var^-}(\Rea(\Phi(u,g)))}_{L^2(C_0,g)}^2 \,dt<\infty.
\eeqs
Hence if $\max(\abs{b^+},\abs{b^-})(t)\to$ 1 as $t\to \infty$ as assumed in the theorem, we may pass to a subsequence of times $t_i\to \infty$ so that 
\beq 
\label{conv:tens-new}
\norm{\De_g u(t_i)}_{L^2(C_0,g)}+\norm{P_g^H(\Rea(\Phi(u,g)(t_i)))}_{L^2(C_0,g)}\to 0
\eeq
or, to be more precise, so that
\beq 
\label{conv:tension-flow}
\de_i:=\norm{\De_g u(t_i)}_{L^2(C_0,g)}+\norm{P_g^{H}(\Rea(\Phi(u,g)(t_i)))}_{L^2(C_0,g)}^4\leq \frac{1}{\sqrt{t_i}}
\eeq
and so that
\beq
\label{conv:b-phi-flow}
b^\pm(t_i)\to b^\pm_\infty, \quad \phi^\pm(t_i)-2\pi n_i \to \phi_\infty^\pm,
\eeq 
where at least one of the $\abs{b_\infty^\pm}=1$, while in the case that $\abs{b_\infty^\pm}<1$, moreover
\beq \label{conv:proj-flow}
\norm{P_g^{\Var^\pm}(\Rea(\Phi(u,g)(t_i)))}_{L^2(C_0,g)}^2\to 0.
\eeq
Theorem \ref{thm:A} then immediately implies that $\ell(t_i)\to 0$ and also yields parts (i) and (ii) of Theorem \ref{thm:1} for the maps $u_i=\tilde u(t_i):(C_{\ell(t_i)},g_{\ell(t_i)})\to \R^n$ defined in \eqref{def:utilde}.

We also note that we can exclude the possibility that there is another sequence of times $t_i'\to \infty$ along which $\ell(t_i')\geq \nu>0$ by applying the above argument to a sequence of nearby times $\tilde t_i$, chosen as follows. Suppose that $\ell(t_i')\geq\nu>0$. As $\abs{\frac{d\ell}{dt}}\leq C\ell^{1/2} \norm{\partial_t g}_{L^2}$, see  \cite[Appendix A.2]{R-cyl}, there exists 
$c(\nu, E_0)>0$ such that  $|\ell(s)-\ell(t)|\leq\tfrac{\nu}{2}$
for all $s,t$ with $|s-t|< c(\nu)$. 
We can thus choose $\tilde t_i$ with $|\tilde t_i-t_i'|\leq c(\nu,E_0)$ and hence $\ell(\tilde t_i)\geq \tfrac{\nu}{2}$, so that the corresponding maps $u(\tilde t_i)$ are almost minimal in the sense that
\eqref{conv:tens-new}, \eqref{conv:b-phi-flow} and \eqref{conv:proj-flow} hold at these $\tilde t_i$ which, by Theorem \ref{thm:A}, leads to a contradiction.

In order to prove that the separation condition \eqref{ass:no-linking} and the choice of the coupling function \eqref{ass:coupling} ensure that 
the obtained boundary bubbles are also conformal and hence indeed minimal,
we first recall that we not only have an orthogonal splitting of the tangent space to $\Mneu$ as described in \eqref{eq:splitting1}, but that, for $\abs{b^+}\neq 0$, the space of variations $\Var^+(g)$ induced by changes of $(b^+,\phi^+)$ furthermore splits $L^2$-orthogonally 
$$\Var^+(g)=\text{span}\{L_{Y_{\abs{b^+}}}g\}\oplus \text{span}\{L_{Y_{\Arg(b^+)}}g\} \oplus\text{span} \{L_{Y_{\phi^+}}g\}$$
into variations induced by changing only one of $\abs{b^+}$, $\Arg(b^+)$, respectively $\phi^+$, where, for example, $Y_{\abs{b^+}}$ is characterised by $L_{Y_{\abs{b^+}}}g=\frac{d}{d\abs{b^+}}h_{b,\phi}^* G_\ell$, for $g=h_{b,\phi}^* G_\ell$.

We may hence control the change of $\abs{b^+}$ by 
\beqs
\abs{\ddt \abs{b^+}}\cdot \norm{L_{Y_{\abs{b^+}}}g}_{L^2}\leq \norm{P_g^{\Var^+}(\pt g)}_{L^2}=\tfrac14 \eta(\abs{b^+})^2 \cdot \norm{P_g^{\Var^+}(\Rea(\Phi(u,g)))}_{L^2}.
\eeqs
As we recall in Appendix \ref{subsect:app-flow}, we have $\norm{L_{Y_{\abs{b^+}}}g}_{L^2(C_0,g)}\geq C(1-\abs{b^+})^{-1}$, see also Lemma 4.3 in \cite{R-cyl},  
so we obtain that if the coupling function $\eta$ satisfies $\eta(\abs{b^+})\leq C(1-\abs{b^+})^\gamma$ for some $\gamma>0$, then
$\ddt (1-\abs{b^+})\leq C (1-\abs{b^+})^{\gamma+1} \cdot 
 \eta_+ \norm{P_g^{\Var^+}(\Rea(\Phi(u,g)))}_{L^2(C_0,g)}$.
Using the energy inequality \eqref{eq:energy-cond}, we may thus bound 
$$\big(1-\abs{b^+})^{-\gamma} \leq C+C \int_0^t \eta_+\norm{P_g^{\Var^+}(\Rea(\Phi(u,g)))}_{L^2(C_0,g)} \leq C\sqrt{t} \qquad \text{ for every } t\geq 1,$$
giving an a priori bound on the speed at which the M\"obius transforms can degenerate of 
\beq \label{est:apriori-eps}
1-\abs{b^+(t)}\geq c \cdot t^{-\frac{1}{2\gamma}}, \qquad c>0.
\eeq
On the boundary, our diffeomorphisms are described by M\"obius transforms whose derivatives are bounded by $\norm{(M_{b,\phi}^{-1})'}_{L^\infty}\leq \frac{C}{1-\abs{b}}$, so we obtain the same control also on the rate at which the three-point condition for our maps $u_i=\tilde u(t_i):(C_i,g_i) \to\R^n$  can degenerate, i.e.~the corresponding points $P^{j,+}(t_i)$ given by \eqref{def:P^j,pm} are so that
$$\min_{k\neq j}(\dist_{g(t_i)}(P^{j,+}(t_i),P^{k,+}(t_i)))\geq c\cdot (1-\abs{b^+(t_i)})\geq c_0 t_i^{-\frac{1}{2\gamma}}
$$
for a constant $c_0$ that depends only on the initial energy and the bound on the coupling function. Of course, the above argument also holds for $\abs{b^-}$.

As the sequence of times $t_i$ was chosen so that \eqref{conv:tension-flow} holds, we hence obtain
 that the  assumption \eqref{ass:new} of Theorem \ref{thm:B} is satisfied as, with the notation from Theorem \ref{thm:B},
$$(1+\abs{\log(\eps_{i,1}/\eps_{i,2})})\cdot \eps_{i,2}^{-1}\cdot \de_i^{1-\frac{1}{q}}\leq C\eps_{i,1}^{-1}\cdot \de_i^{1-\frac{1}{q}}\leq C\cdot t_i^{\frac1{2\gamma}}\cdot t_i^{-\half({1-\frac{1}{q}})} \to 0
$$
as \eqref{ass:relation} allows us to choose $q<\sqrt{2}$ so that $\frac{1}{1-\frac{1}{q}}<\gamma$.

We finally remark that if the three-point condition degenerates on both boundary curves, then both integrals converge
$\int_{\{\pm Y_i^\pm\}\times S^1} u_i\, dS\to \bar q^\pm\in \Gamma^\pm$ 
so that the difference of these integrals is bounded (for $i$ large) by, for example,  $\half \dist(\Gamma^+,\Gamma^-)>0$. The other main assumption \eqref{ass:mv-different} of Theorem \ref{thm:B} is hence 
satisfied in this case for arbitrary disjoint curves $\Gamma^\pm$ not necessarily satisfying the separation condition \eqref{ass:no-linking}.

It is hence only in situations where a boundary bubble forms on only one of the boundary curves, say on $\{-Y_i^-\}\times S^1$ while $\eps_i^+\nto 0$, that we shall use the separation condition \eqref{ass:no-linking}.
In this case, we know from Theorem \ref{thm:B} that the maps $u_i^+(s,\th):=u_i(Y_\ell(t_i)-s,\th):[0,Y_\ell(t_i))\times S^1\to\mathbb{R}^n$
converge to a limit map $u_\infty^+:[0,\infty)\times S^1\to\mathbb{R}^n$ that can be extended across the puncture to a parametrisation of a
minimal disc $S^+$ that spans $\Gamma^+$. 
Letting 
$\de_0:=\dist(\Gamma^-, S^+)>0,$
we then choose $\La>0$ large enough so that 
$\osc_{\{\La\}\times S^1} u_\infty^+<\frac12 \de_0$
and recall that the convergence 
$u_i^+(s,\th)\to u_\infty^+$ 
is in particular uniform on 
such a circle in the interior of $(0,\infty)\times S^1$.
Moreover, the trace of $u_i\vert_{\{-Y_\ell(t_i)\}\times S^1}$ converges locally uniformly to a constant map $q^-\in \Gamma^-$ away from a point, compare Remark \ref{rmk:Courant-Leb}. 
For $Y_i^-:=Y_{\ell(t_i)}$ and $Y_i^+:=Y_{\ell(t_i)}-\La$, we thus obtain that
$$\lim_{i\to \infty} \babs{\int_{\{-Y_i^-\}\times S^1} u_i - \int_{\{Y_i^+\}\times S^1} u_i }= \babs{2\pi q^--
\int_{ \{\La\}\times S^1}u_\infty}\geq \pi \de_0,$$
 so that also in this case assumption \eqref{ass:mv-different} of Theorem \ref{thm:B} is satisfied.  

In both situations, we may thus appeal to this theorem to conclude that in the setting of Theorem \ref{thm:1}, any boundary bubble that is formed by the flow  \eqref{eq:fluss-metric}, \eqref{eq:fluss-map} will again be a minimal immersion, and hence that the flow changes any initial data into either a minimal cylinder or into two minimal discs. 
This completes the proof of Theorem \ref{thm:1}.

\section{Compactness of almost meromorphic functions}
\label{sec:compact}
The goal of this section is to establish the compactness results for almost meromorphic functions stated in Lemma \ref{lemma:key_holo} and Corollary \ref{cor:key-lemma}. 
Our proof of Lemma \ref{lemma:key_holo} is based on the arguments used in the proof of the $L^1$-compactness result \cite[Lemma 2.3]{RT} of Topping and the first author for $C^1$ functions with bounded antiholomorphic derivatives, though we need to proceed with more care as we are dealing with functions that are only `almost meromorphic' rather than `almost holomorphic'. To this end, we use

\begin{lemma}
\label{lemma:mollified}
Let $
\phi\in L^1(\Om)$, $\Om\subset \C$ open, be so that there is a finite set of points $S= \{p^j, j=1\ldots N\}\subset \Om$ such that $\phi\in W^{1,1}_{loc}\big(\Om\setminus S)$ and suppose that $\norm{\dbar \phi}_{L^1(\Om)}<\infty$.
Then for any $\Om'\subset\subset\Om$ and any function $r:\Om'\to (0,\half \dist(\partial \Om,\Om'))$ which is so that 
$$\inf_{z\in \Om',\,p^j\in S} \dist(p^j,\partial D_{r(z)}(z))>0,$$
the function 
\beq
\label{def:moli}
\phi^{(r)}(z):= \fint_{\partial D_{r(z)}(z)}\phi \,ds-\sum_{j=1}^N\frac{\res_{p^j}(\phi)}{p^j-z}\sigma^j(z), \qquad  \res_{p^j}(\phi):=\frac{1}{2\pi \i}\lim_{\ep\to 0} \int_{\partial D_\eps(p^j)}\phi\, dz,
\eeq
approximates $\phi$ in the sense that 
\beq 
\label{est:L1_cauchy}
\norm{\phi-\phi^{(r)}}_{L^1(\Om')}\leq 4\pi \sup_{\Om'} r \cdot \|\dbar\phi\|_{L^1(\Om)}.
\eeq
Here $\sigma^j(z)$ denotes the winding number of $\partial D_{r(z)}(z)$ around $p^j$.
 \end{lemma}

The proof of this lemma  is based on the following standard result.

\begin{lemma} \label{lemma:cauchy}
In the setting of Lemma \ref{lemma:mollified},
the standard Cauchy-formula holds, i.e.
\beq 
\label{eq:Cauchy} 
\int_{\partial \Om'} \phi\, dz=\int_{\Om'} \dbar\phi\, dz\wedge d\bar z+2\pi\i\sum_{p^j\in \Om'}\res_{p^j}(\phi)\eeq
for any $\Om'\subset\subset \Om$ with smooth boundary, 
and there are sequences $\de_i^j\to 0$ such that
\beq 
\int_{\Om'}\bigg|\frac{\res_{p^j}(\phi)}{w-p^j}
-\frac{1}{2\pi \i}\int_{\partial D_{\de_i^j}(p_i^j)}\frac{\phi(z)}{w-z}\,dz\bigg|\, dv_w \rightarrow 0 \text{ as } 
 i\rightarrow\infty.
\label{est:inhom_Cauchy}
\eeq
\end{lemma}
For the sake of completeness, we include a sketch of the proof of \eqref{est:inhom_Cauchy} in the appendix. 
\begin{proof}[Proof of Lemma \ref{lemma:mollified}] 
Let $r$ and $\phi^{(r)}$ be as in the lemma. 
Set $r^*:= \inf_{\Om', j} \dist(p^j,\partial D_{r(z)}(z))>0$, let $\eps\in (0,\half r^*)$ be any fixed number, and set $\Om'_\eps:=\Om'\setminus \bigcup D_\eps(p^j)$. We remark that for $z\in\Om'$ and any such $\eps$,
we have
$\partial D_{r(z)}(z)\subset 
\Om_\eps'$ and 
note that this allows us to conclude that, for any 
$\de^j\in (0,\eps)$,
 \beqa\label{est:L1-1234}
 & \int_{\Om'_\eps}\bigg|\phi(w)-\fint_{\partial D_{r(w)}(w)}\phi\, ds+
 \sum_{j=1}^N \frac{\si^j(w)}{2\pi \i}\int_{\partial D_{\de^j}(p^j)}\frac{\phi(z)}{z-w}\,dz\bigg|
 \leq 4\pi\sup_{\Om'} r 
\norm{\dbar \phi}_{L^1(\Om\setminus\bigcup D_{\de^j}(p^j))}.
  \eeqa
  Indeed for $C^1$ functions this is a simple consequence of the inhomogeneous Cauchy-formula
and Fubini's theorem, while for functions $\phi$ as in the lemma, we can approximate $\phi$ by $C^1$ functions in $W^{1,1}(\Om \setminus\bigcup D_{\de^j}(p^j))$  and use that all integrals in \eqref{est:L1-1234} are  taken over subsets of the set $\Om \setminus\bigcup D_{\de^j}(p^j)$
to pass to the limit.

Finally, 
applying 
\eqref{est:L1-1234} for $\de^j= \de_i^j$ given by  Lemma \ref{lemma:cauchy}, we see that \eqref{est:L1_cauchy} holds true, initially with the $L^1$-norm on the left hand side computed only over $\Om'_\eps$ instead of $\Om'$, but as $\eps\in (0,\half r^*)$ is arbitrary, hence indeed also for the original $L^1(\Om')$-norm.
 \end{proof}
With this result in hand, we 
can now complete the proof of the compactness result for `almost holomorphic functions'. The argument below is based on the proof of Lemma 2.3 in \cite{RT}.
\begin{proof}[Proof of Lemma \ref{lemma:key_holo}]
Let $\phi_i$ be a sequence of functions as in Lemma \ref{lemma:key_holo}, and define
\beqs
\tilde \phi_i:=\phi_i-\sum_{j=1}^N \res_{p_i^j}(\phi_i)\frac{1}{z-p_i^j}
\eeqs
to obtain a function with zero residues at the singular points $p_i^j$ where, after passing to a subsequence, we may assume that $p_i^j\rightarrow p_\infty^j$ (where the $p_\infty^j$ are not necessarily distinct).

Given any $\Om'\subset\subset \Om$ and any $\eps>0$, the main step of the proof is now to establish that  
there exists a sequence of functions $\tilde\phi_i^\ep$
such that, for a constant $C$ independent of $\eps$,
\beqa \label{claim:moll-fun}
\|\tilde\phi_i-\tilde\phi_i^\ep\|_{L^1(\Om')}\leq C\ep,\text{ and }
(\tilde\phi_i^\ep) \text{ is precompact in $L^1(\Om')$ for each $\eps>0$.}
\eeqa
These $\tilde\phi_i^\eps$ will be obtained as mollifications of the $\tilde\phi_i$, where  we will mollify at different scales at different points and with mollifiers supported on annuli rather than on balls due to the singularities of the functions $\tilde\phi_i$, resulting in piecewise smooth functions $\tilde\phi_i^\eps$, rather than smooth functions obtained by a standard mollification.

To prove \eqref{claim:moll-fun},
it suffices to consider $\ep>0$ such that $\min(\dist(p_\infty^k,p_\infty^j))>16\ep$ for distinct $p_\infty^k, p_\infty^j$, and $i$ sufficiently large so that $\max_j|p_i^j-p_\infty^j|<\frac{\ep}{2}$ and we set  
$\Om'_\ep:=\Om'\setminus\bigcup D_\ep(p_\infty^j)$. 
 
We let $\eta\in C_0^\infty(D_2\setminus D_1, [0,\infty))$ be radially symmetric  with $\int_{\R^2}\eta=1$ and let 
$\eta^\de(x):=\frac{1}{\de^2}\eta\big(\frac{x}{\de}\big)$ so that $\supp(\eta^\de)\subset D_{2\de}\setminus D_\de$.
 We now approximate $\tilde \phi_i$ by mollifying at two different scales, namely define
 \beqs
 \tilde\phi^\ep_i:=\tilde \phi_i\ast\eta^{\tfrac{\ep}{8}}\cdot \mathbb{1}_{\Om_\eps'}+\tilde \phi_i\ast\eta^{4\ep}\cdot \mathbb{1}_{\Om'\setminus \Om_\eps'}
 \eeqs
 which corresponds to taking an 
 average of the functions 
 $(\tilde \phi_i)^{(\la r(z))}=(\phi_i)^{(\la r(z))}$
that were defined by \eqref{def:moli} in Lemma \ref{lemma:mollified}. To be more precise, if we set $r(z):=\frac{\ep}{8}$ if $z\in \Om_\ep'$ while $r(z):=4\eps$ if $z\in \Om'\setminus \Om_\ep'$, then we have  
\beqas
\tilde\phi^\ep_i(z)&=2\pi\int_{1}^{2}\eta(\la) (\tilde \phi_i)^{(\la r(z))} \la \,d\la .
 \eeqas We note that, by construction,  
 $\dist(p^j_i,\partial D_{\la r(z)}(z))\geq \frac{3\ep}{4}$
 for every $z\in\Om'$, $\la\in[1,2]$ and $i$ sufficiently large. We may thus  apply Lemma \ref{lemma:mollified}
 for each of the functions $z\mapsto \la r(z)$, to conclude that 
  \beqas
 \norm{\tilde\phi_i^\eps-\tilde\phi_i}_{L^1(\Om')} &= 2\pi
\int_{\Om'}\bigg|  \int_1^2 \eta(\la) \big[(\phi_i)^{(\la r(z))}(z)-\phi_i(z)\big] \la d\la\bigg| dv_z \leq C \sup_{z\in\Om'}r(z) \cdot \norm{\dbar \phi_i}_{L^1(\Om)}
 \eeqas
as claimed in \eqref{claim:moll-fun}.
The second claim of \eqref{claim:moll-fun} follows from standard compactness arguments applied to the sequence $(\tilde\phi^\ep_i)$ on the domains $\Om_\ep'$ and $\Om'\setminus \Om_\ep'$ respectively, as on these domains, the functions $\tilde\phi^\ep_i$ are uniformly bounded in $C^1$ for every fixed $\eps$.
 
Having thus established \eqref{claim:moll-fun}, we note that compactness of the sequence $\tilde \phi_i$ now follows by a standard diagonal sequence argument. Furthermore, the fact that the limit is holomorphic in the case $\norm{\dbar \phi_i}\to 0$ can be obtained exactly as in \cite[Lemma 2.3]{RT}, as the maps $\tilde\phi_i$ have zero residues. 
 \end{proof}
 
\begin{proof}[Proof of Corollary \ref{cor:key-lemma}]
We first remark that since the functions $\phi_i$ are real on the boundary, they can be extended by reflection $\phi(-s+\i\th)=\overline{\phi(s+\i\th)}$ to 
$2\pi \i$ periodic functions on the strip $\{z=s+\i\th\in\C: \abs{s}< \La\}$ which are still in $W^{1,1}$ away from the singular points. 
We may then apply Lemma \ref{lemma:key_holo} on a neighbourhood of 
a fundamental domain, say of $V=(-\La,\La)\times [-\pi,\pi]$ if none of the limit points coincide with $\pm\i\pi$, 
to obtain that for every subsequence of the $\phi_i$, there exists a further subsequence for which the modified functions $\tilde \phi_i$ defined in Lemma \ref{lemma:key_holo} converge strongly in $L^1(\Om')$, where $\Om'=(-\La',\La')\times [-\pi,\pi]$, $\La'\in(0,\La)$ any fixed number. 
We then note that the uniform bound on $\norm{\phi_i}_{L^1(\Om')}+\norm{\tilde \phi_i}_{L^1(\Om')}$ implies that also the functions 
$\sum_j \res_{\i\theta_i^j}(\phi_i)\frac{1}{z-\i \th_i^j}$ have bounded $L^1$-norm which in turn yields a uniform bound on the sum(s)
$$\sum_{l:\th_\infty^l=\th_\infty^j} \res_{\i\th_i^l}(\phi_i).$$
Passing to a subsequence so that the above sum(s) converge, we hence obtain a sequence of functions, obtained as modifications by $\widetilde M_i(z)=\sum_j \frac{a_i^j }{z-\i\th_i^j}$, $a_i^j$ as in the corollary, which converges
$$\phi_i-\widetilde M_i\to \psi_\infty \text{ strongly in }L^1(\Om')$$
to a limit 
$\psi_\infty$ which is meromorphic with poles of order at most one at $\i\th_\infty^j$.

As we already know that $\phi_i$ itself converges to $\phi_\infty$ away from $S_\infty$, it hence suffices to show that $\phi_\infty=\psi_\infty$, or equivalently that
\beq \label{claim:tilde-M-to-0}
\widetilde M_i\to 0 \text{ locally in } \Om'\setminus S_\infty,
\eeq
and that 
\beq \label{claim:M-tilde-M}
\norm{\widetilde M_i-M_i}_{L^1(\Om')}\to 0
\eeq
where $M_i(z)=\sum a_i^j h_{\i\th_i^j}(z)$ is as in the corollary.

In the case that all limit points $\i\th_\infty^j$ are distinct, 
these claims trivially hold true as $\widetilde M_i\equiv M_i=0$. 
In the case that two or three of the limit points coincide, we remark that since the corresponding sums $\sum_ja_i^j=0$ and since the residues must be purely imaginary, we may set $w=\i(z-\i\th_i^2)$, and instead consider functions 
of the form 
$$\tilde m(w)=b_1\cdot\big(\tfrac1{w-\de_1}-\tfrac{1}{w})+b_2\big( \tfrac{1}{w}-\tfrac1{w+\de_2}),$$
where $\de_{1,2}\in (0,\eps)$ and $b_{1,2}\in\R$,
respectively 
$$m(w)=b_1\cdot\big(\tfrac1{2\tan(\half(w-\de_1))}-\tfrac1{2\tan(\half w)}\big)+b_2\big( \tfrac1{2\tan(\half w)} -\tfrac1{2\tan(\half(w+\de_2))}\big), $$
 and show that the analogues of \eqref{claim:tilde-M-to-0} and \eqref{claim:M-tilde-M} hold true for any sequence of such functions $m_i,\tilde m_i$ for which $\norm{\tilde m_i}_{L^1}$ is bounded and for which the maximally allowed distance $\eps_i$ between the points tends to zero. 

To see this, we note that a short calculation, obtained by rewriting $\tilde m$ as 
$$ \tilde m(w)=\tfrac{1}{w(w-\de_1)(w+\de_2)}\cdot [(b_1\de_1+b_2\de_2)w+(b_1-b_2)\de_1\de_2 ] $$
and considering the norm of the function on suitable annuli and sectors around $0$, implies that 
 $$\abs{b_1\de_1+b_2\de_2}\leq C \abs{\log\eps}^{-1}\text{ while } \abs{(b_1-b_2)\de_1\de_2 }\leq C\max(\de_1,\de_2)\leq C \eps$$
 which immediately yields \eqref{claim:tilde-M-to-0}. 
 
These estimates furthermore imply that $\abs{b_{1,2}}\de_{1,2}\leq C$ and hence give that
 $$\abs{\tilde m(w)-m(w)}\leq \abs{b_1\de_1+b_2\de_2}\cdot \abs{f'(w)}+[\abs{b_1}\de_1^2+\abs{b_2}\de_2^2] \sup_{z\in\Om'} \abs{f''(z)}\leq C\abs{\log(\eps)}^{-1}+C \eps,$$
$f$ the holomorphic function $f(w)=\frac1w-\frac1{2\tan(\half w)}$. Hence $\tilde m-m$ converges to zero not just in $L^1$ as needed in \eqref{claim:M-tilde-M} but indeed uniformly on $\Om'$ as $\eps\to 0$. 
\end{proof}

\section{Controlling the residues of the Hopf-differential}
\label{sec:residues}
The goal of this section is to establish the control on the residues of the Hopf-differentials of our sequence of maps claimed in Lemma \ref{lemma:key-res} and on the length of the central geodesic claimed in Lemma \ref{lemma:ell-upper} that were used in the proof of Theorem \ref{thm:B}. 

The proofs of both of these lemmas use the following standard angular energy estimates.

\begin{lemma}
\label{lemma:ang-energy}
For arbitrary maps $u\in H^1(C,\R^n)$ from a hyperbolic cylinder $(C,g)$ as in 
\eqref{def:cyl+-}, the following estimates on the angular energy $\vartheta(s):=\int_{\{s\}\times S^1} \abs{u_\th}^2 d\th$ hold true. 
\\
For any number $q<\sqrt{2}$, we have 
\beq
\label{est:ang-energy-s}
\vartheta(s)\leq C_q E_0e^{-q\min(Y^+-s,s+Y^-)}+C_q\int_{-Y^-}^{Y^+}\int_{S^1} e^{-q\abs{s-t}} \abs{\Delta_{g_0} u}^2 \,d\th\,dt.
\eeq
In particular, 
\beq \label{est:ang-energy-weight}
\int_{-Y^-+1}^{Y^+-1}\int_{S^1} \rho^{-2}\big(\abs{u_\th}^2+\abs{u_{s\th}}^2\big) \,d\th\, ds\leq C\norm{\tens_g u}_{L^2(C,g)}^2+CE_0
\eeq
and, 
 for every $0<\La<\min(Y^+,Y^-)-1$,
 \beq \label{est:u-s-theta}
 \int_{-\La}^\La \int_{S^1} \abs{u_\th}^2+\abs{u_{s\th}}^2 \,d\th\, ds\leq C \rho^2(\La) \norm{\tens_g u}_{L^2(C,g)}^2+CE_0 e^{-q(\min(Y^+,Y^-)-\La)}.
 \eeq
 Here and in the following, $g_0=ds^2+d\th^2$ is the flat metric on the cylinder, $E_0$ denotes an upper bound on the energy of $u$, and the constants $C$ and $C_q$ depend on the numbers $c_{1,2},\ell_0>0$ in the definition of the cylinder, and of course on $q$ for $C_q$.
\end{lemma}
Similar estimates on $\vartheta$ can be found e.g. in \cite{HRT} and \cite{Topping-bubbles}. In addition we also need 
\begin{lemma}
\label{lemma:Mu-osc}
Let $(C,g)$ be a cylinder as described in \eqref{def:cyl+-} and
let $u\in H^1(C,\R^n)$ be a map for which \eqref{ass:MV-lower} is satisfied for some $\alpha>0$. Then $M_u(s):=\fint_{\{s\}\times S^1}u\, d\theta$ satisfies
 \beq \label{est:Mu-osc}
 \abs{M_u'(s)-M_u'(0)}\leq\frac{1}{\sqrt{2\pi}}\rho^{1/2}(s)\,\norm{\tens_g u}_{L^2(C,g)},
 \eeq
 and so in particular
 \beq \label{est:Mu0-lower}
 \abs{M_u'(0)}\geq c\ell \alpha-C\ell^{1/2}\norm{\tens_g u}_{L^2(C,g)}.
 \eeq
 Moreover, we have upper bounds of 
  \beq
 \label{est:Mu0-upper2}
 \abs{M_u'(0)}\leq C \ell^{1/2} +C\ell^{1/2}\norm{\De_g u}_{L^2(C,g)},
 \eeq
 and, if  $\abs{M_u(Y^+)-M_u(-Y^-)}\leq R$ as in Theorem \ref{thm:B}, then also 
  \beq
 \label{est:Mu0-upper}
 \abs{M_u'(0)}\leq C\ell R +C\ell^{1/2}\norm{\De_g u}_{L^2(C,g)}.
 \eeq
 The constants $C,c>0$ depend only on the constants $c_{1,2},\ell_0$ in the definition of the cylinder in \eqref{def:cyl+-} and an upper bound $E_0$ on the energy.
\end{lemma}

Based on these two lemmas,
 which will be proven at the end of this section, we can now give
\begin{proof}[Proof of Lemma \ref{lemma:ell-upper}]
Let $u$ be as in the lemma and let $\Phi=\Phi(u,g)$. Then, by definition,
 \beqas
 \norm{P^H_g(\Rea(\Phi))}_{L^2}=&\abs{\langle \Rea(\Phi),\,\frac{\Rea(dz^2)}{\norm{\Rea(dz^2)}_{L^2}}\rangle}
 =\frac{2}{\norm{\Rea(dz^2)}_{L^2}}\, \big|\int_{-Y_-}^{Y^+}\int_{S^1}(|u_s|^2-|u_\th|^2)\rho^{-2}\,d\th\,ds\big|.
\eeqas 
where we continue to compute all norms and inner products over $(C,g)$ unless indicated otherwise.

To obtain a lower bound on this quantity, we estimate
\beqa\label{tortoise}
\abs{u_s}^2&=|M_u'(0)|^2+ (u_s-M_u'(0))(u_s-M_u'(0)+2M_u'(0))
\geq \tfrac12 |M_u'(0)|^2- C\abs{u_s-M_u'(0)}^2\\
&\geq \tfrac12 |M_u'(0)|^2- C\abs{u_s-M_u'(s)}^2-C\abs{M_u'(0)-M_u'(s)}^2
\eeqa
and recall that
 \beq
 \label{est:dz-squared-upper-lower}
 c \ell^{-3/2}\leq\norm{\Rea(dz^2)}_{L^2}\leq C \ell^{-3/2}
 \eeq
holds true for some constants $C,c>0$ that depend only on the fixed upper bound $ \ell_0$ on $\ell$ and the fixed numbers $c_{1,2}$ from the definition of the cylinder in \eqref{def:cyl+-}.

We may thus bound 
\beqa\label{est:initial-est-for-Yl}
\norm{P^H_g(\Rea(\Phi))}_{L^2}
&\geq \frac{c}{2}\ell^{-3/2} |M_u'(0)|^2
-I
 \eeqa
for a remainder term $I$
that we can bound, using Lemmas \ref{lemma:ang-energy} and \ref{lemma:Mu-osc},  
as 
\beqas 
I&:=
C \ell^{3/2}\int \big[\abs{u_\th}^2+\abs{u_s-M_u'(s)}^2 +\abs{M_u'(0)-M_u'(s)}^2\big]\rho^{-2} d\th\,ds\\
&\leq C \ell^{3/2}\int_{-Y^-+1}^{Y^+-1}\int_{S^1} (\abs{u_\th}^2+\abs{u_{s\th}}^2)\rho^{-2}\, d\th\, ds +C \ell^{3/2} E_0 
+ C \ell^{3/2} \norm{\Delta_g u}_{L^2}^2\cdot \int \rho^{-1}\,ds 
\\
&\leq  C \ell^{3/2}\cdot (E_0+(1+C\ell^{-2} )\norm{\Delta_g u}_{L^2}^2) \leq C\ell^{3/2}+ C\ell^{-1/2} \norm{\Delta_g u}_{L^2}^2,
 \eeqas
where we used that $\int\rho^{-1}\le C\ell^{-2}$ in the penultimate step.

Inserting this bound into \eqref{est:initial-est-for-Yl} and using 
the lower bound \eqref{est:Mu0-lower} on $\abs{M_u'(0)}$
thus yields
 \beqas
 \norm{P^H_g(\Rea(\Phi))}_{L^2}\geq C \alpha^2 \ell^{1/2}-C\ell^{3/2}-C\ell^{-1/2}\norm{\tens_g u}_{L^2}^2
 \eeqas
as claimed in \eqref{est:hor-lower}.
 So, provided $\bar \ell=\bar \ell(\al,c_1,c_2,E_0)>0$ is chosen sufficiently small, we have that for $\ell\leq \bar \ell$ 
 \beqs
\ell \leq C\big(\norm{P^H_g(\Rea(\Phi))}_{L^2}^2+\ell^{-1} \norm{\tens_g u}_{L^2}^4\big),
 \eeqs
 which easily implies the estimate \eqref{est:upper-ell} claimed in the lemma. Indeed, if 
$\ell^{-1}\norm{\tens_g u}_{L^2}^2\leq 1$, then 
\eqref{est:upper-ell} is a direct consequence of the above estimate, while for $\ell\leq \norm{\tens_g u}_{L^2}^2$ the claimed bound \eqref{est:upper-ell} is trivially satisfied.
\end{proof}
As a next step towards the proof of Lemma \ref{lemma:key-res}, we show
 \begin{lemma}\label{lemma:res-sum-est}
 In the setting of Lemma \ref{lemma:key-res}, the residues of the function $\phi$ describing the Hopf-differential $\Phi=\Phi(u,g)$ may be controlled by
 \beq
 \label{est:sum-res}
 \abs{\sum\res_{P^j}(\phi)}\leq C \big[ \norm{\tens_g u}_{L^2(C,g)}+\exp(-c/\norm{P^H_g(\Rea(\Phi))}_{L^2(C,g)}^2) +\ell^2\norm{\tens_g u}_{L^2(C,g)}^2\big],
 \eeq
while, for any $0<\Lambda\leq L:=Y^++Y^-$ and $q<\sqrt{2}$, furthermore
 \beq
 \label{est:weight-sum-res}
 \abs{\sum e^{\i\th^j}\res_{P^j}(\phi)}\leq C_q e^\Lambda\big(\norm{\tens_g u}_{L^2(C,g)}+ \norm{\tens_g u}_{L^2(C,g)}^2+\norm{P^H_g(\Rea(\Phi))}_{L^2(C,g)}^4+e^{-q\Lambda}\big)
 \eeq
and hence, for $\norm{\De_g u}_{L^2}+\norm{P_g^H(\Rea(\Phi))}_{L^2}\leq M$ for some $M<\infty$,
\beq
\label{est:main-step-res} 
 \abs{\sum e^{\i\th_j}\res_{P^j}(\phi)}\leq C_{q,M} \big[\norm{\tens_g u}_{L^2(C,g)}^{1-1/q}+\norm{P^H_g(\Rea(\Phi))}_{L^2(C,g)}^{4(1-1/q)}
 \big]\eeq 
where as in Lemma \ref{lemma:key-res} the constants $c,C,C_q>0$ depend only on $\alpha>0$, upper bounds $E_0$ and $R$ on the energy and on $\abs{M_u(Y^+)-M_u(-Y^-)}$ and as usual the constants $c_{1,2}$ in the definition of the cylinder, respectively additionally on $q$ (and $M$) for $C_q$ (respectively $C_{q,M}$).
\end{lemma}

\begin{rmk} \label{rmk:no-upper-bound}
It is useful to observe that the first claim \eqref{est:sum-res} of this lemma still remains valid
if we drop the assumption of an upper bound $R$ on  $\abs{M_u(Y^+)-M_u(-Y^-)}$, while the only adjustment needed for the other claims \eqref{est:weight-sum-res} and \eqref{est:main-step-res} is to replace $\norm{P_g^H(\Rea(\Phi))}_{L^2}^4$ by $\norm{P_g^H(\Rea(\Phi))}_{L^2}^2$.
\end{rmk}

\begin{proof}[Proof of Lemma \ref{lemma:res-sum-est}]

As the function $u$ satisfies the stationarity condition on the boundary curve $\{-Y^-\}\times S^1$ (except at the points $P^j$), we know that the trace of the Hopf-differential on this curve is real so that we can extend 
$\phi$ by $\phi(-Y^--s,\th):=\overline{\phi(-Y^-+s,\th)}$ to a meromorphic function on the cylinder $(-Y^--L,Y^+)\times [-\pi,\pi]$  with poles only at $P^j$. 

We will prove the first two claims of the lemma by applying the Cauchy formula \eqref{eq:Cauchy}  for $\phi$ and for $e^{z}\phi$ on suitable domains and estimating the resulting terms using Lemmas \ref{lemma:ang-energy} and \ref{lemma:Mu-osc}.

To prove \eqref{est:sum-res} we apply \eqref{eq:Cauchy}
for $[-2Y^-+\lambda,-\lambda]\times[-\pi,\pi]$ and integrate over $\lambda \in [0,1]$ to get 
 \beqas
 2\pi\abs{\sum\res_{P_j}(\phi)}\leq& 2\int_{-Y^-}^0\int_{S^1}|\dbar\phi|\,d\th\,ds+\bigg|\int_{0}^1\int_{S^1}\phi(-\lambda,\th)-\bar\phi(-\lambda,\th)\,d\th\,d\la\bigg|\\
 \leq& C\norm{\tens_{g_0} u}_{L^2(C,g_0)}+C\bigg|
 \int_{[-1,0]\times S^1}\langle u_s,u_\th\rangle \,d\th\,ds\bigg|\\
 =& C\norm{\tens_{g_0} u}_{L^2(C,g_0)}+C\bigg| \int_{[-1,0]\times S^1}\langle u_s-M_u'(s),u_\th\rangle\, d\th\,ds\bigg|\\
 \leq & 
  C\norm{\tens_g u}_{L^2(C,g)}+C\bigg(\int_{-1}^0\vartheta(s)\,ds\bigg)^{1/2}\cdot\bigg( \int_{[-1,0]\times S^1} \abs{u_{s\th}}^2 \,d\th\, ds\bigg)^{1/2}
 \eeqas 
where the last step follows by using Wirtinger's inequality as well as the uniform upper bound $\rho\leq C(c_1,\ell_0)$ on the conformal factor.

Thanks to Lemma \ref{lemma:ang-energy},
we may thus bound for every $q<\sqrt{2}$,
 \beqas
\abs{\sum\res_{P_j}(\phi)}\leq&
  C\big(\norm{\tens_g u}_{L^2}+e^{-q Y^-}+\ell^2\norm{\tens_g u}_{L^2}^2\big)
 \eeqas 
 where here and in the following norms are computed over $(C,g)$ unless indicated otherwise. 
 If $\norm{P^H_g(\Rea(\Phi))}_{L^2}\geq \norm{\tens_g u}_{L^2}$, then Lemma \ref{lemma:ell-upper} tells us that $Y^-\geq c\ell^{-1}\geq c\norm{P^H_g(\Rea(\Phi))}_{L^2}^{-2}$.
 If instead $\norm{\tens_g u}_{L^2}\geq\norm{P^H_g(\Rea(\Phi))}_{L^2}$, the same lemma yields that $Y^-\geq c\norm{\tens_g u}_{L^2}^{-2}$ and so certainly $e^{-q Y^-}\leq C\norm{\tens_g u}_{L^2}$. 
 In either case, we have that 
 \beqs
 \abs{\sum\res_{P_j}(\phi)}\leq C\big(\norm{\tens_g u}_{L^2}+e^{-c/\norm{P^H_g(\Rea(\Phi))}_{L^2}^2}+\ell^2\norm{\tens_g u}_{L^2}^2\big). \eeqs
 This concludes the proof of \eqref{est:sum-res}.
 
 To prove \eqref{est:weight-sum-res}, we argue similarly, integrating now  $\phi \,e^z$, $z=s+\i\th$,   over the rectangles  $[-Y^--\Lambda-\lambda,-Y^-+\Lambda+\lambda]\times[-\pi,\pi]$, $\lambda\in[-1,0]$, where $\Lambda\in (0, Y_--1)$ is any fixed number.
Writing for short $I_\La:= [-Y^-+\La-1,-Y^-+\La]$, this yields
  \beqa \label{est:hallo-cookies}
 \abs{\sum\res_{P_j}(\phi)\cdot e^{\i\th_j}}\leq&C e^\Lambda\norm{\tens_g u}_{L^2}+2\int_{I_\La\times S^1}e^s \abs{\phi(s,\th)} \,d\th\,ds\\
\leq & C e^\Lambda \norm{\De_g u}_{L^2}
+Ce^\La \int_{I_\La} \vartheta(s)\,ds +
 Ce^\La \int_{I_\La\times S^1}\abs{u_s}^2 \,d\th\,ds\\
 \leq & C e^\Lambda\norm{\tens_g u}_{L^2}+Ce^{-(q-1)\La}+Ce^\La \norm{\De_g u}_{L^2 }^2+C e^\La J,
 \eeqa
 where we applied Lemma \ref{lemma:ang-energy} in the last step and where
 $J:=\int_{I_\La\times S^1} \abs{u_s}^2 \,d\th\,ds.$
 As in \eqref{tortoise}, we now split
 \beqs
\abs{u_s}^2\leq C |M_u'(0)|^2 + C\abs{u_s-M_u'(s)}^2+C\abs{M_u'(0)-M_u'(s)}^2
\eeqs
into terms that can be bounded using Lemmas \ref{lemma:ang-energy}, \ref{lemma:Mu-osc}. Combined with Lemma \ref{lemma:ell-upper} we thus get
 \beqa \label{est:squirrel}
J \leq& C\int_{I_\La\times S^1}|u_{s\th}|^2 +|M_u'(s)-M_u'(0)|^2\,d\th\,ds+C(R^2 \ell^2+\norm{\tens_gu}_{L^2}^2)\\
 \leq&C\big(\norm{\tens_g u}_{L^2}^2+\norm{P^H_g(\Rea(\Phi))}_{L^2}^4+e^{-q\Lambda}\big),
 \eeqa
where $C=C(E_0,\al, c_{1,2}, R)$. 
Inserted into \eqref{est:hallo-cookies} this gives the second claim \eqref{est:weight-sum-res} of the lemma.

To obtain the final claim, we first note that if $\de:=\norm{\tens_g u}_{L^2}+\norm{P^H_g(\Rea(\Phi))}_{L^2}^4$
is no less than some fixed  number $\de_0=\de_0(c_{1,2},\ell_0,E_0)>0$ chosen later, then we can choose $\La=1$ and obtain \eqref{est:main-step-res} immediately from \eqref{est:weight-sum-res}.
Conversely, for $\de<\de_0$ we may choose $\Lambda=-1/q\log\de$ so that 
$e^\Lambda\de=e^{\Lambda(1-q)}= \de^{1-\frac{1}{q}}$ as, for $\de_0=\de_0(c_{1,2},\ell_0,E_0)$ suitably small,
Lemma \ref{lemma:ell-upper} ensures that $L\geq c\sqrt{\de }\geq \La$. 
Hence, by \eqref{est:weight-sum-res}, 
$$
\babs{\sum e^{\i\th^j} \res_{P^j}(\phi)}\leq C_q  \de (1+\norm{\De_g u}_{L^2}) e^{\La}+ Ce^{
(1-q)\La}\leq C_{q,M}  \de^{1-\frac{1}{q}}$$
which gives the final claim of the lemma.
\end{proof}

We note that the only adjustment needed in the above argument in a situation where we drop the assumption of a uniform upper bound $R$ on $\abs{M_u(Y^+)-M_u(-Y^-)}$ as considered in Remark \ref{rmk:no-upper-bound}
is that we have a  weaker upper bound on $\abs{M_u'(0)}$ now given by \eqref{est:Mu0-upper2} instead of \eqref{est:Mu0-upper} so that in \eqref{est:squirrel} and the subsequent estimates we need to replace $\norm{P^H_g(\Rea(\Phi))}_{L^2}^4$ by $\norm{P^H_g(\Rea(\Phi))}_{L^2}^2$.

Based on this lemma, we can now finally establish Lemma \ref{lemma:key-res}
which played a key role in the proof of Theorem \ref{thm:B}.
\begin{proof}[Proof of Lemma \ref{lemma:key-res}]
We first remark that \eqref{claim:sum-res} was already proven in Lemma \ref{lemma:res-sum-est}. To prove the other claims, we  may assume without loss of generality that 
$\th^2=0$, so $\th^1=-\eps_1$, $\th^3=\eps_2$ for $0<\eps_1<\eps_2<2\pi-(\eps_1+\eps_2)$. 
We set
$$
 A_{\eps_1,\eps_2}
 =\begin{pmatrix} 
1 & 1 & 1 \\
\cos(-\eps_1) & 1 & \cos(\eps_2) \\
\sin(-\eps_1) & 0 & \sin(\eps_2)
\end{pmatrix} 
 $$ 
 and note that Lemma \ref{lemma:res-sum-est} gives bounds of $ C_{q,M}  \de^{1-\frac1q}$ for each of the components of the vector $A_{\eps_1,\eps_2}\cdot (\res_{P^1}(\phi), \res_{P^2}(\phi),\res_{P^3}(\phi))^T$ since the trace of $\phi$ on $\{-Y^-+\i\th\}\times S^1$ is real, and hence the residues
are purely imaginary. Here and in the following, $q<\sqrt{2}$ and $ \de:= \norm{\De_g u}_{L^2}+\norm{P^H_g(\Rea(\Phi))}_{L^2}^4$.

As 
 $\abs{\det(A_{\eps_1,\eps_2})}\geq c\eps_1\eps_2^{2}>0$ and 
 $$
 A_{\eps_1,\eps_2}^{-1}=\det(A_{\eps_1,\eps_2})^{-1} 
 \cdot \begin{pmatrix} 
\sin(\eps_2) & -\sin(\eps_2) & \cos(\eps_2)-1 \\
-\sin(\eps_1+\eps_2) & \sin(\eps_1)+\sin(\eps_2) & -\cos(\eps_2)+\cos(\eps_1) \\
\sin(\eps_1) & -\sin(\eps_1) & -\cos(\eps_1)+1
\end{pmatrix} ,
$$
we have 
$$\abs{(A_{\eps_1,\eps_2}^{-1})_{ij}}\leq C\eps_1^{-1} \eps_2^{-1} \text{ for } i=1,2 \text{ and } j=1,2,3$$
while 
$$\abs{(A_{\eps_1,\eps_2}^{-1})_{ij}}\leq C\eps_2^{-2} \text{ for } i=3 \text{ and } j=1,2,3$$
so the claimed estimate \eqref{claim:res} on the residues immediately follows. 

To obtain  \eqref{claim:eps-sum-res}
we finally note that, by 
Lemma \ref{lemma:res-sum-est},  
$\vert\sin(-\eps_1)\res_{P^1}(\phi) +\sin(\eps_2)\res_{P^3}(\phi)\vert\leq C \de^{1-\frac1q}$, so we obtain that 
 \beqas
 \abs{\eps_2\res_{P^3}(\phi)-\eps_1\res_{P^1}(\phi)}\leq & 
 \abs{\sin(\eps_2)\res_{P^3}(\phi)-\sin(\eps_1)\res_{P^1}(\phi)}+
 C\eps_1^3 \abs{\res_{P^1}(\phi)}+C \eps_2^3\abs{\res_{P^3}(\phi)}\\
 &\leq C\de^{1-\frac1q} 
 \eeqas
 as claimed in the lemma. 
 \end{proof}

For the sake of completeness, we finally include proofs of the auxiliary Lemmas \ref{lemma:ang-energy} and \ref{lemma:Mu-osc}.

\begin{proof}[Proof of Lemma \ref{lemma:ang-energy}]
We follow closely the arguments of \cite[Lemma 3.7]{RT3} and \cite[Lemma 2.5]{HRT}, which further simplify as our target manifold is $\mathbb{R}^n$. As in these papers, we first note that 
\beqs
 \vartheta''(s)=2\int_{\{s\}\times S^1}|u_{s\th}|^2+|u_{\th\th}|^2-u_{\th\th}\cdot\tens_{g_0} u. 
\eeqs
For any fixed $q<\sqrt{2}$, we thus have, applying  Wirtinger's inequality in the second step,
\beqas
 \vartheta''(s)\geq& q^2\int_{\{s\}\times S^1}|u_{\th\th}|^2-C_q\int_{\{s\}\times S^1}|\tens_{g_0} u|^2
 \geq q^2\,\vartheta(s)-C_q\int_{\{s\}\times S^1}|\tens_{g_0} u|^2
\eeqas
so that the first claim \eqref{est:ang-energy-s} follows from the  maximum principle.

To bound the weighted integral of $\vartheta$ considered in \eqref{est:ang-energy-weight} we observe that 
 $\int e^{-q\abs{s-t}} \rho^{-2}(s)\, ds\leq C \rho^{-2}(t)$ for any $s,t\in[-Y^-,Y^+]$. 
Integrating \eqref{est:ang-energy-s} and using Fubini thus allows us to bound  
\beqa \label{est:ice-cream}
 \int\rho^{-2}|u_\th|^2\,d\th\,ds\leq& CE_0 (\rho^{-2}(Y^+) +\rho^{-2}(-Y^-))+C\int \rho^{-2}\abs{\De_{g_0}u}^2\, dv_{g_0}
 \\
 \leq&CE_0+C\norm{\tens_gu}_{L^2(C,g)}^2
\eeqa
as claimed, 
where integrals are always computed over the whole of $C$ unless indicated otherwise
and where 
we used in the last step that the conformal factor is bounded away from zero uniformly at the ends of the cylinder. 
We note that the same argument, using now also that $\rho$ is increasing in $\abs{s}$, implies the bound on the integral of $\vartheta$ claimed in \eqref{est:u-s-theta}. 

To obtain the claimed estimates for $u_{s\th}$, we 
let $\eta=\eta(s)\in C_c^{\infty}([-\La-1,\La+1],[0,1])$ be a smooth cut-off function, chosen so that $\eta \equiv 1$ on $[-\La,\La]$ and $\norm{\eta'}_{L^{\infty}}\leq 2$. As $\abs{(\rho^{-2})'}\leq 2\rho^{-1}\leq C\rho^{-2}$, we can then bound 
 \beqas
 I&:=\int\eta^2\rho^{-2}|u_{s\th}|^2\,d\th\,ds
 =-\int u_\th\cdot\partial_s(\eta^2\rho^{-2}u_{s\th})\,d\th\,ds\\
  &\leq \half I +
  C \int\rho^{-2}|u_\th|^2d\th\,ds -\int u_\th u_{ss\th}\rho^{-2}\eta^2\,d\th\,ds,
 \eeqas
so that, by \eqref{est:ice-cream},
 \beqas
I&\leq CE_0+C\norm{\De_g u}^2_{L^2(C,g)}-2\int u_\th u_{ss\th}\rho^{-2}\eta^2\,d\th\, ds \\
&=CE_0+C\norm{\De_g u}^2_{L^2(C,g)} -2 \int u_\th\cdot \partial_\th (\De_{g_0} u-u_{\th\th}) \rho^{-2} \eta^2 \,d\th\,ds 
\\
&=CE_0+C\norm{\De_g u}^2_{L^2(C,g)} +2\int \big( u_{\th\th} \Delta_{g_0} u - \abs{u_{\th\th}}^2) \rho^{-2}\eta^2 \,d\th\,ds
\\
&\leq CE_0+C\norm{\De_g u}^2_{L^2(C,g)}+\int \abs{\De_{g_0} u}^2 \rho^{-2} \,d\th\,ds =C E_0+ C\norm{\De_g u}_{L^2(C,g)}^2
\eeqas
as claimed in 
 \eqref{est:ang-energy-weight}.
 Finally, the estimate for $\int \abs{u_{s\th}}^2$ 
 may be obtained by the same argument, which in this situation simply yields an additional factor of $\sup_{\abs{s}\leq \La+1}\rho^2(s)\leq C\rho^2(\La)$. 
\end{proof}

We finally include a brief proof of the auxiliary Lemma \ref{lemma:Mu-osc}.
\begin{proof}[Proof of Lemma \ref{lemma:Mu-osc}]
As $\int_0^s \rho^{2}=\frac{\ell}{2\pi} \tan(s)\leq \rho(s)$ and 
$M''_u(s)= (2\pi)^{-1}\int_{\{s\}\times S^1}\tens_{g_0} u$ 
 we have that for every $s_0\in[-Y^-,Y^+]$
 \beqas
 \abs{M_u'(s_0)-M_u'(0)}=
 (2\pi)^{-1}\bigg|\int_0^{s_0}\int_{S^1}\tens_{g_0} u\,d\th\,ds\bigg|
 \leq (2\pi)^{-1/2}\rho(s_0)^{1/2}
 \norm{\tens_g u}_{L^2(C,g)},
 \eeqas
 as claimed. 
  To show the second claim, we observe that the assumption \eqref{ass:MV-lower} ensures that
 \beqs
 \alpha\leq\bigg|\int_{-Y^-}^{Y^+}M_u'(s)\,ds\bigg|\leq L|M_u'(0)|+\int_{-Y^-}^{Y^+}|M_u'(s)-M_u'(0)|\,ds
 \eeqs
so that the claimed lower bound on $\abs{M_u'(0)}$ follows from the fact that 
$L=Y^++Y^-\leq \frac{c}{\ell}$, 
the above estimate, and the fact that 
 \beqas
 \int_{-Y^-}^{Y^+}\rho^{1/2}ds \leq 2 \bigg(\frac{\ell}{2\pi}\bigg)^{-1/2}
 \int_0^{\frac{\pi}{2}}\cos^{-1/2}(t)\,dt \leq
 C\ell^{-1/2}.
 \eeqas
 By the same argument, an upper bound on $\abs{M_u(Y^+)-M_u(-Y^-)}$ yields \eqref{est:Mu0-upper}, while more generally we can bound 
  \beqas
 L \abs{M_u(0)}\leq(2\pi)^{-1}\babs{ \int u_s \,d\th\, ds}+\int_{-Y^-}^{Y^+}|M_u'(s)-M_u'(0)|\,ds
 \leq CE_0 L^{1/2}+C\ell^{-1/2}\norm{\De_gu}_{L^2},
 \eeqas
 so that \eqref{est:Mu0-upper2} holds. 
\end{proof}

\appendix
\section{ }
\subsection{Remarks on the flow}
\label{subsect:app-flow}
$ $\\
In this appendix, we explain why the results of \cite{R-cyl} remain valid for general coupling functions as considered in the present paper. 

We first note  
as in \cite{R-cyl}, that also for general coupling functions $\eta_\pm$, the evolution equation for the metric reduces to a system of 7 ordinary differential equations (coupled with the equation \eqref{eq:fluss-map} for the map). Short-time existence of solutions can hence be obtained exactly as in \cite[Section 3]{R-cyl} and solutions exist for as long as $\ell$ remains bounded away from zero and $(b,\phi)$ remain in a compact region of the parameter domain.

We then recall that the variations of the metric induced by changes of the parameters $(b^\pm,\phi^\pm)$ are supported in 
fixed compact regions of $C^\pm$, to be more precise in $\{\pm s\in [\half,1]\}$. Hence the delicate argument of \cite[Lemma 4.4]{R-cyl} (which is based on \cite{RT3}) that prevents $\ell\to 0$ in finite time applies without change also for the flow \eqref{eq:fluss-metric}, \eqref{eq:fluss-map} with general coupling functions, 
 as this analysis is carried out only on the central part of the cylinder.

The argument of \cite[Lemma 4.1]{R-cyl} that $(b,\phi)$ remains in a compact set of $\Om^2$ also applies with only minor changes as $\eta$ is assumed to be bounded above. Indeed, we recall that the generating vector field $Y_{\abs{b^+}}$ is so that 
\beq\label{eq:wolf}
\norm{L_{Y_{\abs{b^+}}}g}_{L^2(C_0,g)}\geq C(1-|b^+|)^{-1}.
\eeq
We note that while we would need this only for $|b^+|$ close to 1, and that this is the range of $\abs{b^+}$ for which \eqref{eq:wolf} was proven in \cite{R-cyl}, a short calculation shows that \eqref{eq:wolf} is indeed true for all values of $|b^+|\in[0,1)$. So, with $\La_\pm:=\sup\eta_\pm$, we have that
\beqas
|\tfrac{d}{dt}|b^+||\cdot\norm{L_{Y_{\abs{b^+}}}g}_{L^2(C_0,g)}\leq& \norm{P^{\mathcal{V}^+}_g(\partial_t g)}_{L^2(C_0,g)}=\tfrac{1}{4}\eta_+^2\cdot\norm{P^{\mathcal{V}^+}_g(\Rea(\Phi(u,g)))}_{L^2(C_0,g)}\\
\leq& \tfrac{\La_+}{4} \eta_+\norm{P^{\mathcal{V}^+}_g(\Rea(\Phi(u,g)))}_{L^2(C_0,g)},
\eeqas
excluding the possibility that $|b^\pm|\to 1$ in finite time, compare also Section \ref{subsec:asympt}. Likewise, the bound on $\frac{d}{dt}\phi^\pm$ in \cite[Lemma 4.1]{R-cyl} changes only by a factor of $\La_{\pm}$.

Having thus explained why the arguments of \cite{R-cyl} yield long-time existence of solutions to \eqref{eq:fluss-metric}, \eqref{eq:fluss-map} for any fixed coupling functions $\eta_\pm$, we finally remark that the asymptotic analysis of \cite[Theorem 2.7]{R-cyl} in the case that the three-point condition does not degenerate is also unaffected by the choice of coupling function as, in this case, $|b^\pm(t_i)|$ is contained in a compact subset of $[0,1)$ so that $\eta_\pm(t_i)$ is bounded away from $0$. Along a sequence of times $t_i$ as considered in \cite[Theorem 2.7]{R-cyl}, we hence still obtain that $$\norm{\Delta_{g_i} u(t_i)}_{L^2}+\norm{P^H_{g_i}(\Rea(\Phi(t_i)))}_{L^2}+\norm{P_{g_i}^{\mathcal{V}^\pm}(\Rea(\Phi(t_i)))}_{L^2}\to0$$ so that the proof of \cite[Theorem 2.7]{R-cyl} applies without change.

\subsection{Remarks on the stationarity condition and on almost meromorphic functions}
\label{subsect:app-stat}
$ $\\
In this part of the appendix, we present some of the properties and uses of the stationarity condition \eqref{eq:stationar} that was used extensively throughout the paper. We first recall that the stationarity condition can be thought of as a weak formulation of the condition that the function $\phi$ defining the Hopf-differential $\Phi=\phi\, (ds+\i d\th)^2$ of a map $u:C_0\to \mathbb{R}^n$ is real on the boundary. Indeed, if $u\in H^2(C_0,g)$, then \eqref{eq:stationar} is equivalent to
$\int_{\partial C_0}\Rea(\Phi)(\frac{\partial}{\partial s},X)\,d\th =0$
by Stokes' Theorem. So, as
$\Rea(\Phi)=\Rea(\phi)(ds^2-d\th^2)-\Ima(\phi)(ds\otimes d\th+d\th\otimes ds),$
the stationarity condition reduces to
$\int_{\partial C_0}\Ima(\phi)X^\th\,d\th=0$.

As all limit maps we obtain are harmonic, it is useful to observe that in this case, the stationarity condition reduces to
\beq\label{eq:harm-stat}
\int L_Xg\cdot\Rea(\Phi) \, dv_g=0,
\eeq
and that we have the following.

\begin{rmk}
\label{lemma:no-poles}
Let $\Phi=\phi\,dz^2$ be a quadratic differential described by a function $\phi\in L^1(D_r^+)\cap W^{1,1}_{loc}(D_r^+\setminus\{0\})$ such that $\dbar\phi=0$. As observed above, if \eqref{eq:harm-stat} is satisfied for all $X\in\Gamma(TD_r^+)$ with $X(0)=0$, then $\phi$ is real on $D_r^+\cap\{s=0\}\setminus\{0\}$, so can be reflected to give a meromorphic function $\phi$ on $D_r$ with only a possible pole at $0$. Moreover, if \eqref{eq:harm-stat} is satisfied for all $X\in\Gamma(TD_r)$, then $\res_0(\phi)=0$, and hence $\phi$ is holomorphic.
\end{rmk}

\begin{proof}
To see that $\res_0(\phi)=0$, we apply \eqref{eq:harm-stat} for $X=\eta\cdot\frac{\partial}{\partial\th}$, where $\eta$ is a cut-off function such that $\eta\equiv1$ on $D_{r/2}^+$. Now Stokes' Theorem and \eqref{eq:harm-stat} imply
 \beqs
 \big|\int_{\{s=\eps\}}X^\th\Ima(\phi)\,d\th\big|=\big|\int_{\{s>\eps\}} L_X g\cdot \Rea(\Phi)\,dv_g\big|=\big|\int_{\{s\leq\eps\}} L_X g\cdot \Rea(\Phi)\,dv_g\big|\to 0 \text{ as $\eps\to 0$.}
 \eeqs
 Conversely, as $\phi$ is integrable, it cannot have a pole of order two or higher, and so we can write
 $\phi(z)=\frac{a\i}{z}+\psi(z)$, where $\psi(z)$ is holomorphic and so bounded. As $\phi$ is real on $\{s=0\}$, we must have $a\in\mathbb{R}$, so $\Ima(\phi)=\frac{as}{|s+\i\th|^2}+\Ima(\psi(z))$ and $\Ima(\psi)|_{\{s=0\}}=0$. Thus 
 \beqa
 \lim_{\eps\to 0}\big|&\int_{\{s=\eps\}}X^\th\Ima(\phi)\,d\th\big|= \lim_{\eps\to0}|a|\big|\int \eta\frac{\eps}{|\eps+\i\th|^2}\,d\th\big|\geq\lim_{\eps\to0} |a|\eps\int_{-\eps}^{\eps}\frac{1}{\th^2+\eps^2}\,d\th=\tfrac{\pi}{2}|a|,
 \eeqa
 hence $\abs{a}=\abs{\res_{0}(\phi)}=0$.
\end{proof}

Finally, for the sake of completeness, we provide 
 \begin{proof}[Sketch of proof of estimate \eqref{est:inhom_Cauchy}
 from Lemma \ref{lemma:cauchy}]

A short calculation, based on the generalised Cauchy formula, establishes that there exists a number $\de_0>0$ such that 
for every $0<\de<\de_0$,
  \beq\label{est:lion}
\int_{\Om'}\bigg|\frac{\res_{p^j}(\phi)}{w-p^j}-
\frac{1}{2\pi i}\int_{\partial D_{\de}(p^j)}\frac{\phi(z)}{w-z}\,dz
\bigg|\,dv_w \leq C\de|\log\de|\int_{\partial D_\de(p^j)}|\phi|\,ds +C|m^j(\de)|,
  \eeq
  where 
  $C$ is a 
  universal constant and where 
  $m^j(\de):=\res_{p^j}(\phi)-\frac{1}{2\pi i}\int_{\partial D_\de}\phi\, dz\to 0$ as $\de\to 0$. Estimate \eqref{est:inhom_Cauchy} is then an immediate consequence of this bound, the integrability of $\de\mapsto\int_{\partial D_\de(p^j)}|\phi|\,ds$ and the fact that $\de \mapsto (\de \abs{\log(\de)} )^{-1}$ is not integrable near zero.
\end{proof}

{\sc Melanie Rupflin and Matthew R. I. Schrecker,\\ 
Mathematical Institute, University of Oxford, Oxford, OX2 6GG, UK}

\end{document}